\documentclass[graybox]{svmult}

\usepackage{amsmath,amsfonts,amssymb,color,verbatim,eucal,mathrsfs}

\usepackage{mathptmx}       % selects Times Roman as basic font
\usepackage{helvet}         % selects Helvetica as sans-serif font
\usepackage{courier}        % selects Courier as typewriter font
\usepackage{type1cm}        % activate if the above 3 fonts are
\usepackage{makeidx}         % allows index generation
\usepackage{graphicx}        % standard LaTeX graphics tool
\usepackage{multicol}        % used for the two-column index
\usepackage[bottom]{footmisc}% places footnotes at page bottom

\makeindex             % used for the subject index
\newtheorem{thm}{Theorem}[section]

\newtheorem{rmk}[thm]{Remark}

\newtheorem{lem}[thm]{Lemma}
\newtheorem{deft}[thm]{Definition}

\newtheorem{ex}{Example}

\newcommand{\R}{\mathbb R}

\newcommand{\N}{\mathbb N}

\newcommand{\supp}{\mathop{\rm supp}}

\newcommand{\ep}{\varepsilon}

\newcommand{\va}{\varphi}
\newcommand{\ppp}{\partial}
\newcommand{\ooo}{\overline}
\newcommand{\OOO}{\Omega}

\newcommand{\xxxj}{x_j}

\newcommand{\weight}{e^{2s\va}}

\newcommand{\dist}{\mathrm{dist}\,} 
\begin{document}

\title*{
%MY Carleman estimates and an application to the observability inequality 
%for transport equations
Observability inequalities for transport equations through Carleman estimates
}
% Use \titlerunning{Short Title} for an abbreviated version of
% your contribution title if the original one is too long
\author{Piermarco Cannarsa, Giuseppe Floridia and Masahiro Yamamoto}
% Use \authorrunning{Short Title} for an abbreviated version of
% your contribution title if the original one is too long

\institute{Piermarco Cannarsa \at 
Department of Mathematics, 
      University of Rome \lq\lq Tor Vergata'',
      00133 Rome, Italy,\\
\email{cannarsa@mat.uniroma2.it}
\and Giuseppe Floridia \at 
Department of Mathematics and Applications \lq\lq R. Caccioppoli'',
      University of Naples Federico II,
%      Dipartimento di Matematica e Applicazioni ÒR. CaccioppoliÓ
%Universitˆ degli Studi Napoli ÒFederico IIÓ
%Via Cintia, Monte S. Angelo
%I-
80126 Naples, Italy, \email{giuseppe.floridia@unina.it}
\and Masahiro Yamamoto \at 
Department of Mathematical Sciences, The University of Tokyo, Komaba, Meguro, 
Tokyo, 153 Japan, \email{myama@ms.u-tokyo.ac.jp}}
%
%
%\institute{Piermarco Cannarsa \at Name, Address of Institute, \email{cannarsa@mat.uniroma2.it}
%\and Name of Second Author \at Name, Address of Institute \email{name@email.address}}
%
% Use the package "url.sty" to avoid
% problems with special characters
% used in your e-mail or web address
%
\maketitle

\abstract{
We consider the transport equation 
$\ppp_t u(x,t) + %\alpha'
H(t)\cdot \nabla u(x,t) = 0$ in $\OOO\times(0,T),$ where $T>0$
%for
%$x\in \OOO$ and $0 < t < T$, where 
and $\OOO\subset \R^d%,\, d\in\N,
$ is a bounded
domain with smooth boun\-dary $\ppp\OOO$.
First, we prove a Carleman estimate for solutions of finite energy with piecewise continuous weight functions.  Then, under a further condition %on $H$
 which guarantees that the 
orbits of $H$
 %$\{%\alpha' H(t)\in\R^d~:~ 0 \le t \le T\}$ 
 intersect $\ppp\OOO$, we prove an energy estimate which in turn yields an obser\-vability inequality.
% using a suitable Carleman estimate. 
 Our results are  motivated by applications to inverse problems.
% Carleman estimate.
%A %usual 
%cut-off argument and such a Carleman estimate 
%yield an energy estimate called an observability inequality.
}

\section{Introduction%and Carleman estimate
}\label{intro}
Let $ d\in\N$ and $\OOO \subset \R^d%, d\in\N,
$ be a bounded domain with smooth boundary 
$\ppp\OOO$, $\nu = \nu(x)$ be the unit outward normal vector 
at $x$ to $\ppp\OOO$, and let $x \cdot y$ and $|x|$ denote the scalar product of 
$x, y \in \R^d$ and the norm of $x\in \R^d,$ respectively.
We set $Q := \OOO\times (0,T),$ and
we consider
\begin{equation}\label{E1}
Pu(x,t) := \ppp_tu + H(t)\cdot \nabla u = 0 \quad 
\mbox{in $Q$},              
\end{equation}
where
$%\alpha'
H(t) := (H_1(t),\ldots, H_d(t)):[0,T]\rightarrow \R^d,$\,
$H\in C^1([0,T];\R^d)$.

Equation (\ref{E1}) is called a transport equation and $H(t)$ describes 
the velocity of the %convection 
flow, which is here assumed to be independent of the spatial 
variable $x$. 
%\begin{ex}\label{sources}
%Setting $$v(x,t) = f(x-\alpha(t))$$ with 
%$f \in C^1(\ooo{\OOO};\R)$ and 
%$\alpha = (\alpha_1, ..., \alpha_d) \in C^2([0,T];\R^d)$, we see that 
%$v=v(x,t)$ satisfies (\ref{E1}) where $H(t) = \alpha'(t)$, $0 \le t \le T$.
%Thus (\ref{E1}) is related to an inverse problem of determining
%a moving source $f$ with given $\alpha$ in the diffusion equation
%$$
%\partial_tw = \Delta w + f(x-\alpha(t)), \quad x\in \Omega, \, 0<t<T.
%$$
%\end{ex}
\subsection*{
%{\bf 
Problem formulation}
We assume 
%$$
%\vert H(t)\vert \neq0,\;\;\;\forall t\in[0,T],%H_0 > 0, 
%\quad 0\le t \le T,  
%$$
%and we set 
\begin{equation}\label{H1}
\displaystyle H_0:=\min_{t\in[0,T]}|H(t)|>0,\end{equation} 
%so that
%\begin{equation}\label{H1}
%\vert H(t)\vert \geq H_0 > 0, \quad \forall t\in [0,T].
%\end{equation}
%
and, without loss of generality, we suppose that ${\bf 0}=(0,\ldots,0)\in\overline{\Omega}.$
%Otherwise we can translate $\Omega$ suitably to reduce to the case of
%${\bf 0} \in \ooo{\OOO}$. 
 \\
%\$
%In the following Lemma \ref{Pr1} we show that the previous nondegenerate vector valued function $H(t)$ admits a partition $\{t_j\}_0^m$ of $[0,T]$ (
%\footnote{A partition $\{t_j\}_0^m$ of $[0,T]$ is a strictly increasing finite sequence $t_0, t_1, \ldots, t_m$ (for some $m\in\N$) of real numbers starting from the initial point $t_0=0$ and arriving at the final point $t_m=T.$\\
%In the following we will call it a {\it uniform partition} when the lenght of the intervals $[t_j,t_{j+1}],\;j=0,\ldots,m-1,$ is constant, that is, $t_{j+1}-t_j=\frac{T}{m}\;\;\;\forall j=0,\ldots,m-1$.}
%) 
%such that the angle of oscil\-lation of the vector $H(t)$ is less than $\frac{\pi}{4}$ in any time interval $[t_{j},t_{j+1}],\;j=0,\ldots,m-1.$
%%\begin{prop}\label{Pr1}
%%Given a vector valued function $H\in Lip([0,T];\R^d)$
%%%(\footnote{$H\in Lip([0,T];\R^d)\iff \exists L>0 : |H(t)-H(s)|_{\R^d}\leq L |t-s|$}) %
%%that satisfies condition 
%%(\ref{H1}), 
%%there exist
%% $m\in\N,\;$ and a partion $\{t_j\}_0^m$ of $[0,T]$
%% ($t_j\in[0,T],\;0:=t_0\leq %t_{j-1}<
%% t_j<t_{j+1}\leq t_m:=T, \; j=0,\ldots,m-1$)
%%%where $\eta_j:=
%%%\frac{H(t_j)}{|H(t_j)|}
%%%,\,j=1,\cdots,m.$
%%%for some $\eta_j\in \R^d$ with $|\eta_j|=1,\;j=1,\cdots,m.$
%%\end{prop}
%($t_j\in[0,T],\;0:=t_0\leq %t_{j-1}<
% t_j<t_{j+1}\leq t_m:=T, \; j=0,\ldots,m-1$)
%\$

Let us recall the following definition.
\begin{deft}
A partition $\{t_j\}_0^m$ of $[0,T]$ is a strictly increasing finite sequence $t_0, t_1, \ldots, t_m$ (for some $m\in\N$) of real numbers starting from the initial point $t_0=0$ and arriving at the final point $t_m=T.$

Hereafter, we will call $\{t_j\}_0^m$ a {\it uniform partition}  of $[0,T]$   when the length of the intervals $[t_j,t_{j+1}]$ is constant for $j=0,\ldots,m-1,$ that is, $%t_{j+1}-t_j=\frac{T}{m}$ $\left(\;
t_j = \frac{T}{m}j,\;\;j=0,\ldots,m%\right)
.$
\end{deft}
Lemma \ref{Pr1} below ensures that any vector-valued 
function $H(t),$ sati\-sfying (\ref{H1}), admits a 
partition $\{t_j\}_0^m$ of $[0,T]$ 
such that the angles of oscillations of the vector $H(t)$ are less than 
$\frac{\pi}{2}$ in any time interval $[t_{j},t_{j+1}],\;j=0,\ldots,m-1$ (see Figure \ref{cones}).\\

Given a partition $\{t_j\}_0^m$ of $[0,T],$ let us set
%Here by a uniform partition $\{t_j\}_0^m$ of $[0,T]$, we mean 
%$t_j = \frac{T}{m}j,\:$ for $j=0,1,..., m$.
%\\
\begin{equation}\label{etaj}
\eta_j:=\frac{H(t_j)}{|H(t_j)|},\;\;\;j=0,\ldots,
m-1.
\end{equation}

\begin{lem}\label{Pr1} Let \,$ S_*\in(1/\sqrt{2},1).$
For any given  $H\in Lip([0,T];\R^d)$,
satisfying condition (\ref{H1}), 
there exist $m\in\N$ and a partition $\{t_j\}_0^m$ of $[0,T]$
such that
\begin{equation}\label{H2}
\frac{H(t)}{\vert H(t)\vert} \cdot
\eta_j%\frac{H(t_j)}{|H(t_j)|}
\geq S_*,
%\frac{1}{\sqrt{2}},
\;\;\;\forall t\in[t_{j},t_{j+1}], \;\;\forall j=0,\ldots,m-1,
\end{equation}
where $\eta_j$ are defined in (\ref{etaj}).
\end{lem}
Lemma \ref{Pr1} is proved in the  Appendix.
%\textcolor{red}{Put here the Picture 1} %Section \ref{SecPr1}.

\begin{figure}[htbp]
\label{cones}
\begin{center}
\hskip-.5cm
\includegraphics[width=7.45cm, height=4.10cm]{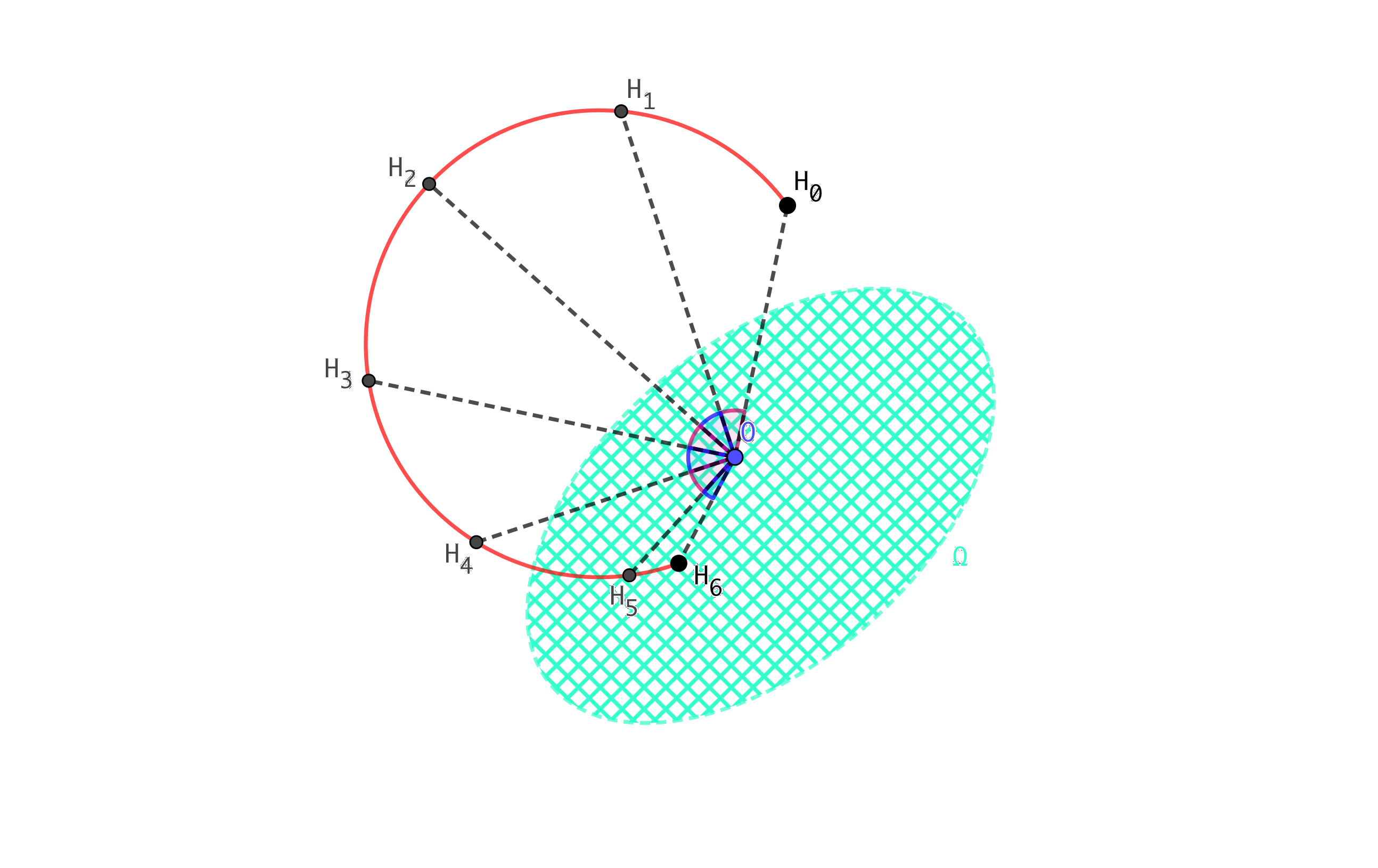}
\caption{In this picture $S_*=\cos\frac{\pi}{6},\:m=6$ and $H_j:=H(t_j),\,j=0,\ldots,5.$}
%\label{cones}
\end{center}
\end{figure}

\begin{rmk}\em\label{rmk:H2}
 Condition (\ref{H2}) means that there exist $m$ cones in $\R^d$ such 
that the axis of every cone, that is, the straight line passing through the 
apex about which the whole cone has a circular symmetry, is the line 
between $O=(0,\ldots,0)$ and $\eta_j,%=\frac{H(t_j)}{|H(t_j)|},
\;j=0,\ldots,
m-1$. Moreover, a straight line passing through the apex is contained in the 
cone if the angle between this line and the axis of the cone is less than 
$\pi/4$. Indeed, the inequality (\ref{H2}), that is 
$H(t) \cdot\eta_j>  \cos\vartheta^*\vert H(t)\vert$ for some $\vartheta^*\in(0,\frac{\pi}{4}),$
is equivalent to the fact that the angle between $H(t)$ and $\eta_j$ is less 
than %or equal to 
$\pi/4.$ 
Thus,  $H(t)$ is contained in the same cone 
$\forall t\in[t_j,t_{j+1}].$ Let us note that it can occur that $\eta_i=\eta_j,$
for $i\neq j.$
\end{rmk}

Let $\delta_{\Omega}=diam(\Omega)=\displaystyle \sup_{x,y\in\overline{\Omega}}|x-y|.$
Let us fix $S_*\in(1/\sqrt{2},1),$ $r>0$ and define 
\begin{equation}\label{xj}
x_j:=-R_j\eta_j,\;\;\;j=0,\ldots,m-1,
\end{equation}
where $\eta_j$ is defined in (\ref{etaj}) and 
\begin{equation}\label{Rj}
\left\{ \begin{array}{rl}
R_j&=2^jR_0+(2^j-1)(\delta_\Omega+r),\\
R_0&= 
\frac{1+S_*}{1-S_*}%(3+2 \sqrt{2})
 \delta_{\OOO}.
%(3+2\sqrt{2})\delta_\Omega .
\end{array}\right.
\end{equation}
%with $R_0\geq (3+2\sqrt{2})\delta_\Omega,$
% fixed real number. 
 We note that from (\ref{Rj}) it follows that
$$x_j\not\in%\Omega
\overline{\OOO},\qquad j=0,\ldots,m-1.$$

\begin{figure}[htbp]
\label{points}
\begin{center}
\hskip-.5cm
\includegraphics[width=12.25cm, height=6.75cm]{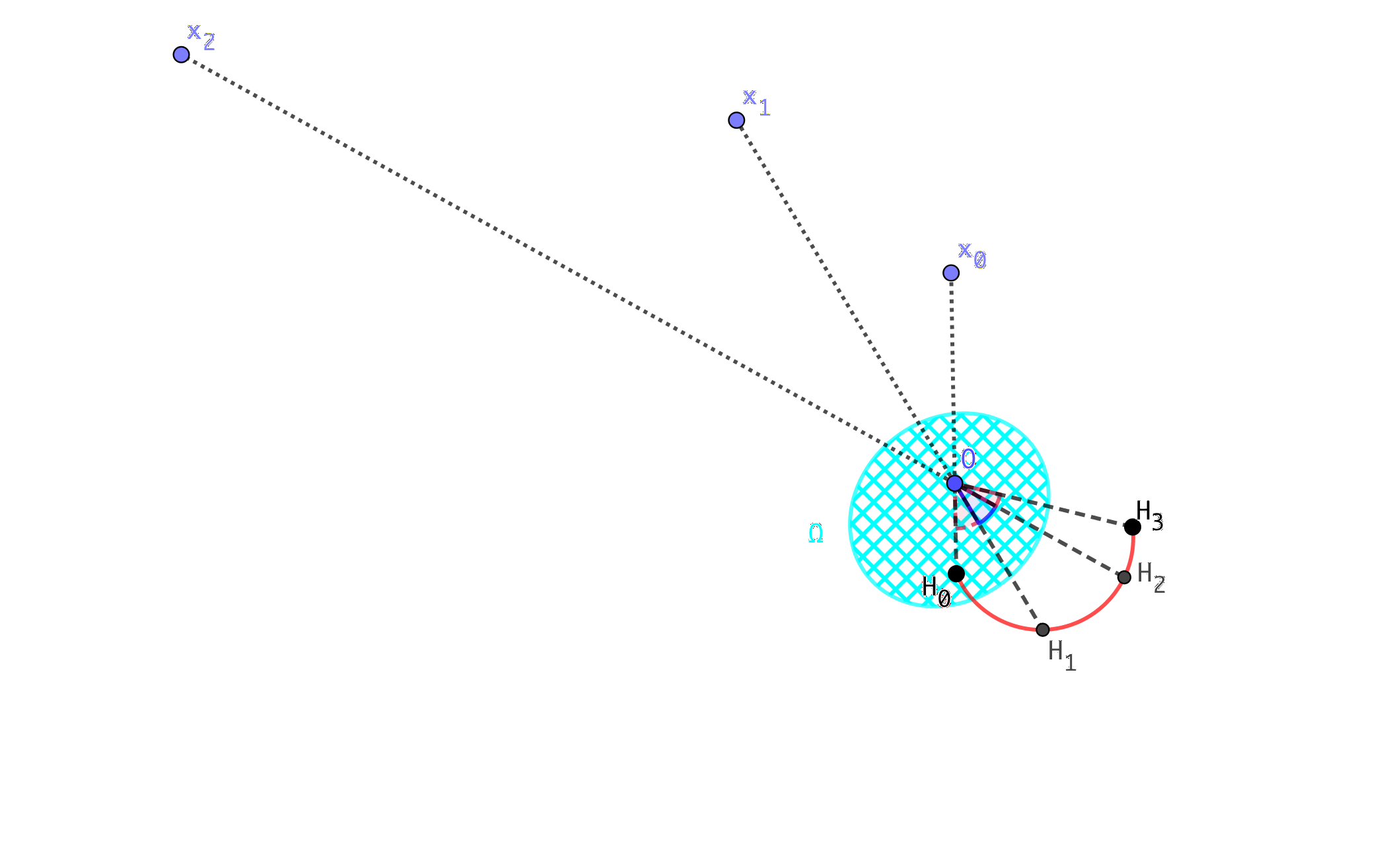}
\caption{In this picture $S_*=\cos\frac{\pi}{6},\:m=3$ and $H_j:=H(t_j),\,j=0,1,2.$}
\label{points}
\end{center}
\end{figure}
For every $j=0,\ldots,m-1,$ let us define
\begin{equation}\label{wf4}
%M_j=
M_\Omega(x_j):=\max_{x\in\overline{\Omega}}|x-x_j|\qquad \mbox{ and }  \qquad
%\mu_j=
d_\Omega(x_j):=\min_{x\in\overline{\Omega}}|x-x_j|.
\end{equation}
%In Section \ref{sec wf} we prove the following Lemma.
%\begin{lem}\label{M<mj}
%Let $\; x_j=-R_j\eta_j,\;\;j=0,\ldots,m-1,$ with $R_j$ defined as in (\ref{Rj}). Then
%\begin{equation}\label{wf3}
%M_\Omega(x_j)=\max_{x\in\overline{\Omega}}|x-x_j|<
%%=:
%\min_{x\in\overline{\Omega}}|x-x_{j+1}|=d_\Omega(x_{j+1})
%%\mu_{j+1}
%,\qquad j=0,\ldots,m-2.
%\end{equation}
%\end{lem}
%\begin{rem}
%\rm We note that the points $x_j, \;j=0,\ldots,m-1,$ are further away from $\bf 0$ as $j$ increases.
%\end{rem}

%A consequence of
%Lemma \ref{M<mj} is that 
\begin{rmk}
 \em 
 The choice of the $R_j$'s in (\ref{Rj}) (see Lemma~\ref{M<mj} below and Figure~\ref{points}) guarantees that the points $x_j$'s are located sufficiently far away from $\OOO$ and %in a position by 
at   in\-creasing distances from the origin.
\end{rmk}

By the choice of the finite sequence $R_j=\vert \xxxj\vert$ in (\ref{Rj}) ($R_j$ sufficiently large 
compared with $\delta_{\OOO}$) 
%enough large respect to $\delta_\Omega,$
we deduce in Lemma \ref{Prx0j} below that %it yields that
$$(x+R_j\eta_j)\cdot\eta_j\ge%\frac{\sqrt{2}}{2}
S_*|x+R_j\eta_j|,\;\; \forall 
x\in \overline{\OOO}.$$
% $x_j=-R_j\eta_j$ with 
%some $R_j>0$ chosen so that $x_j\not\in\overline{\OOO}$ and
%\begin{equation}\label{cone}
%\angle C_{min}(\xxxj) \le\frac{\pi}{4}\,.
%\end{equation}
%Here $C_{min}(\xxxj)$ denotes 
In other words, the apex angle of the minimum cone 
with the apex $\xxxj$ which includes $\OOO$ is less than $2\arccos S_*(<\pi/2)$ (see Figure~\ref{angle}).

\begin{figure}[htbp]
\label{angle}
\begin{center}
\hskip-.5cm
\includegraphics[width=10.7cm, height=5.9cm]{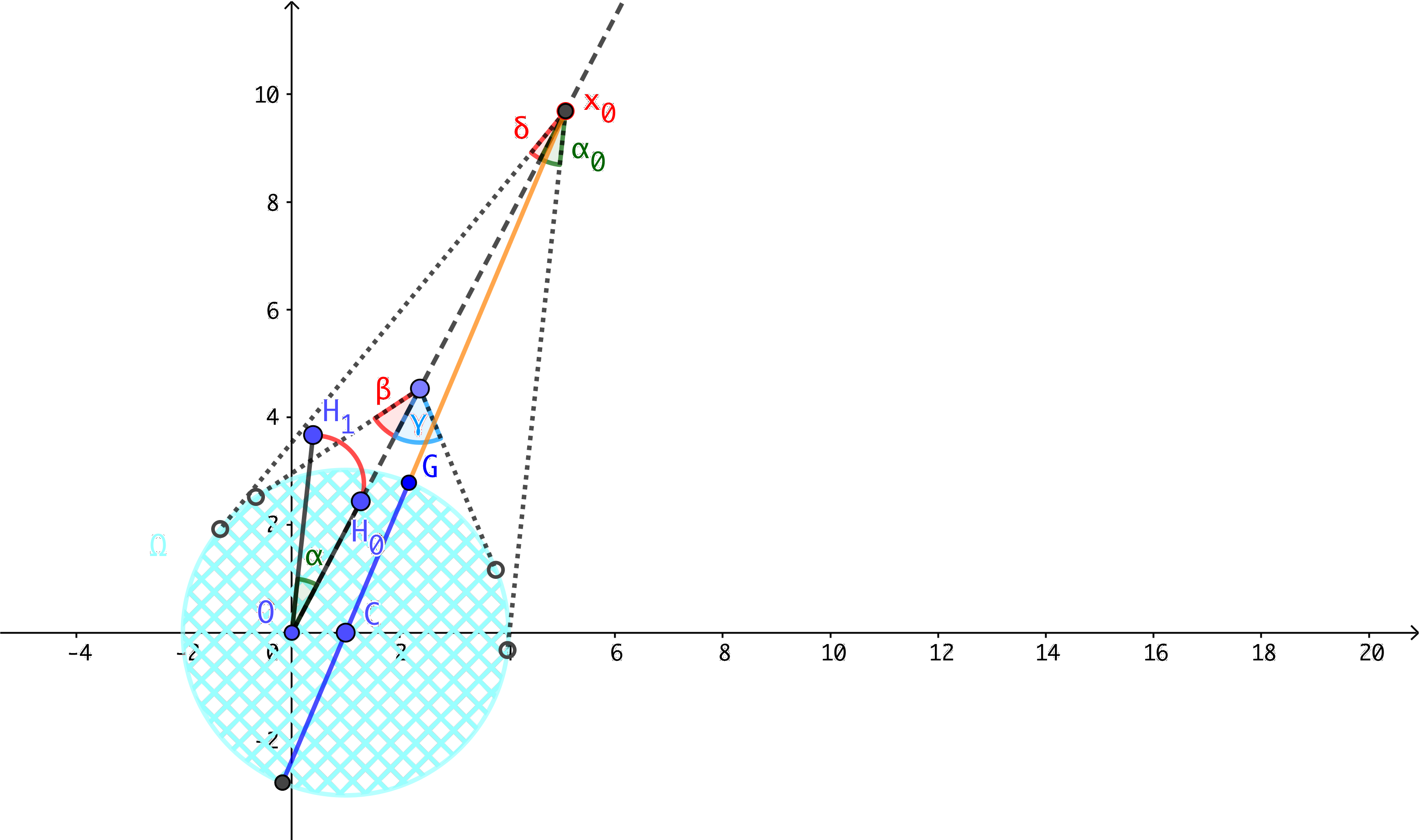}
\caption{In this picture: $\OOO := \{(x,y)\in R^2:\vert (x,y)-(1,0)\vert <3\},\:C=(1,0),$\, $S_*=\cos\alpha\in(\frac{1}{\sqrt{2}},1),\:m=1,$ $H_j:=H(t_j),\,j=0,1,$  and $\beta,\gamma>\alpha,\,\alpha_0=\alpha, \delta\leq\alpha.$ We note that $d_\Omega(x_0)=\dist(x_0,G)$ and $M_\Omega(x_0)=d_\Omega(x_0)+6.$}
\label{angle}
\end{center}
\end{figure}

\bigskip
We now introduce the weight function $\va(x,t)$, to be used in our Carleman estimate, as follows. First, we define $\va$ on $%Q = 
\ooo{\OOO}\times [0,T)$ setting, for every $x\in%\Omega
\ooo{\OOO},$ 
\begin{equation}\label{wf}
\va(x,t) = \va_j(x,t) := -\beta (t-t_j) + \vert x-\xxxj\vert^2, \;
\;t\in [t_{j}, t_{j+1}), \; j=0,\ldots, m-1,% \eqno{(1.8)}
\end{equation}
where 
\begin{equation}\label{Hdot}
\beta:=%(\sqrt{2}-1)
(2S_*^2-1)H_0d_\Omega(x_0),
\end{equation}
 with $H_0$ and $d_\Omega(x_0)$ defined by (\ref{H1}) and $(\ref{wf4}),$ respectively. Then we extend $\va$ to $%Q = 
\ooo{\OOO}\times [0,T]$ by continuity. Observe that $\va$ is piecewise smooth in $t$ and smooth in $x$.
%will be chosen later, $x_j=-R_j\eta_j$ with 
%some $R_j>0$ chosen so that $x_j\not\in\overline{\OOO}$ and

\subsection*{Main results
%Carleman estimate with a piecewise continuous
%weight function
}
In this article, under condition (\ref{H1}), we establish an observability 
inequality for (\ref{E1})
which estimates the $L^2$-norm of $u(x,0)$ by lateral boundary data
$u\vert_{\ppp\OOO\times (0,T)}$ under some conditions on 
$H(t)$ (see Theorem \ref{thm obs}). This observability inequality %, showed in Section \ref{Sec obs}, 
is a consequence of the following %a particu\-lar 
Carleman estimate.
% with a piecewise 
%continuous
%weight function, obtained in the following Theorem \ref{Carleman}.
\begin{thm}\label{Carleman}%[Carleman estimate]
Let $u\in H^1(Q)$ be a solution of equation (\ref{E1}),
%such that 
%$u(\cdot,T)=0,$ 
where $H\in C^1([0,T];\R^d)$ satisfies (\ref{H1}).
%Let us consider the equation (\ref{E1}) under the condition (\ref{H1}).
Let $\{t_j\}_0^m$ be a partition of $[0,T]$ satisfying (\ref{H2}). 
%given by Proposition \ref{Pr1}.
 %such that holds (\ref{H2}). 
 Then, there exist constants $s_0, C_0, C > 0$ %and $C_0, C>0$ 
such that for all $s > s_0$ we have
 \begin{eqnarray}%\label{CarlIne}
s^2\!\!\!\!\!\!\!\!\!&&\int_Q\vert u\vert^2 e^{2s\va} dxdt
+ se^{-C_0s}\sum_{j=0}^{m-1}\int_{\OOO} \vert u(x,t_j)\vert^2 dx
                                      \nonumber\\
&\le& C\int_Q \vert Pu\vert^2 \weight dxdt
+ Cse^{Cs}%\sum_{j=1}^m 
\int_{\Sigma%_j
} \vert u\vert^2 d\gamma dt+ Cse^{%2
 C s
 %M_*^2
 }\int_{\OOO}\vert u(x,T)\vert^2\,dx, \nonumber
%\label{Carl}
\end{eqnarray}
where $\varphi(x,t):Q\longrightarrow\R$ is the %smooth 
weight function defined in (\ref{wf}), and
\begin{equation}\label{eq:subb}
\Sigma = \{ (x,t) \in \ppp\OOO \times(0,T):\,
\thinspace H(t)\cdot \nu(x) \ge 0 \}
\end{equation}
is the subboundary of all exit points for $H$.
\end{thm}

%\subsection*{Application to the observability inequality}\label{Sec obs}

%Let $g\in L^2(\ppp\OOO\times (0,T))$ and let us consider the following problem
%\begin{equation}\label{pb obs}
%\left\{ \begin{array}{rl}
%& \ppp_tu + H(t)\cdot \nabla u = 0 \quad 
%\mbox{in $Q:= \OOO\times (0,T)$},\\
%& u\vert_{\ppp\OOO \times (0,T)} = g.
%\end{array}\right.
%\end{equation}
%{\textcolor{red}{
%After introducing some further notations, 
We now give the
%Then, let show the statement of an 
observability inequality for the equation (\ref{E1}).
%that permits to obtain the uniqueness for the corresponding inverse problem by measurement on the boundary.
\begin{thm}\label{thm obs}
Let $g\in L^2(\ppp\OOO\times (0,T))$ and let us consider the following problem
\begin{equation}\label{pb obs}
\left\{ \begin{array}{rl}
& \ppp_tu + H(t)\cdot \nabla u = 0 \quad 
\mbox{in $Q:= \OOO\times (0,T)$},\\
& u\vert_{\ppp\OOO \times (0,T)} = g.
\end{array}\right.
\end{equation}
Let us suppose that there exists a partition $\{t_j\}_0^m$ of $[0,T]$ associated to $H(t)$ satisfying (\ref{H2}) such that the following condition holds
\begin{equation}\label{max cond} 
\max_{0\le j\le m-1}\frac{(t_{j+1}-t_j)d_\Omega(x_j)}{M^2_\Omega(x_j)}>\frac{1}{H_0%(\sqrt{2}-1)
(2S_*^2-1)},
\end{equation}
where $M_\Omega(x_j), d_\Omega(x_j)$ and $H_0$ are defined in $(\ref{wf4})$ and $(\ref{H1}),$ respectively.
%There exists $T_0>0$ such that if $T>T_0,$
%{\bf Theorem 3.1.}\\
%{\it
%then there exists a 
Then, there exists
a constant $C>0$ such that the following inequality holds %such that 
$$
\Vert u(\cdot,t)\Vert_{L^2(\OOO)} \le 
C\Vert g\Vert_{L^2(\ppp\OOO\times (0,T))}, \quad 0\le t \le T,
$$ 
for any $u \in H^1(Q)$ satisfying (\ref{pb obs}). 
%with $g\in 
%L^2(\ppp\OOO\times (0,T))$.
\end{thm}
%The presence in the statement of Theorem \ref{thm obs} of the threshold time $T_0$ requires some important comments that are given in Section \ref{threshold time}, but 

Assumption  (\ref{max cond}) is meant to guarantee that the 
orbit $\{%\alpha'
H(t)\in\R^d~:~t\in[0,T]
 %0 \le t \le T
 \}$ 
 intersects $\ppp\OOO.$
In the following example, we show that this or a similiar condition is indeed necessary: observability fails without some extra assumption.\\

In the following, for $\eta>0$ we consider $\Omega_\eta:=\{z\in\R^2:\,|z|<\eta\}.$
\begin{ex}\label{cex}\rm
Let $\sigma>0$ and $\rho\in(0,2\sigma/3).$ Let $\Omega:=\Omega_{\rho}$ and let $f\in  C^1(\ooo{\OOO}_\sigma;\R)$ be such that $\supp(f)\subset \Omega_{\rho/2}\subseteq \ooo{\OOO}_\sigma %\in C^1(\ooo{\OOO};\R)
$ and let $\alpha(t)%=(\alpha_1(t),\alpha_2(t)) 
= (\rho\cos t, \rho\sin t),\,t\in[0,2\pi].$
We set 
$$v(x,y,t) = f(x-\rho\cos t,y-\rho\sin t).$$
Thus, % with 
% \in C^2([0,T];\R^2)$, 
$v$ satisfies (\ref{E1}), where $H(t) = \alpha'(t)$, $0 \le t \le T$, and $v$ vanishes at the boundary of $\OOO_\sigma$. So,
\begin{equation}\label{spt0}
\left\{ \begin{array}{rl}
& \ppp_tv + \alpha'(t)\cdot \nabla v = 0 \quad 
\mbox{in $%Q:= 
\OOO_\sigma\times (0,T)$},\\
& v\vert_{\ppp\OOO_\sigma \times (0,T)} = g,
\end{array}\right.
\end{equation}
with $g\equiv0.$
We note that $|\alpha'(t)|=\rho>0$ and, for $t\in[0,T]$, the support of $v(\cdot,\cdot,t)$ is %the following
\begin{equation}\label{spt}
%\left\{ \begin{array}{rl}
\supp(v(\cdot,\cdot,t))=%\Omega_{\frac{\rho}{2}}(\rho \cos t, \rho\sin t):
\left\{(x,y)\in\R^2:\,\left|(x-\rho \cos t, y-\rho\sin t)\right|<\frac{\rho}{2}\right\}.
%\end{array}\right.
\end{equation}
%Thus (\ref{E1}) is related to an inverse problem of determining
%a moving source $f$ with given $\alpha$ in the diffusion equation
%$$
%\partial_tw = \Delta w + f(x-\alpha(t)), \quad x\in \Omega, \, 0<t<T.
%$$
Then, from (\ref{spt0}) and (\ref{spt}) it follows that %cannot hold 
 observability fails.\hfill $\square$
% of Theorem \ref{thm obs} because a contradiction it occurs, that is %since for any constant $C>0,$
%$$
%\forall\, C>0\;\;\;\; 0<\Vert v(\cdot,\cdot,t)\Vert_{L^2(\OOO_\sigma)} \le 
%C\Vert g\Vert_{L^2(\ppp\OOO_\sigma\times (0,T))}=0, \quad 0\le t \le T.
%$$ 
%It show that the assumption (\ref{max cond}) is essential in Theorem \ref{thm obs}.
\end{ex}

%\textcolor{red}{Put here the Picture 4}\\
%\vspace{0.2cm}

%\textcolor{red}{
%{\bf Interpretation of (\ref{large time}).}\\
%In general, $m \in \N$ can be large.
%For example we consider $\OOO = \{\vert x\vert < R\}$.  Then we can see that 
%$M_0 = \mu_0 + 2R$.  By (14), we see that 
%$\mu_1 > M_0 = \mu_0 + 2R$ and $M_1 = \mu_1 + 2R$.  Hence
%$M_1 > \mu_0 + 2(2R)$.
%Continuing this, we obtain $M_{m-1} > \mu_0 + 2Rm$.  
%Therefore (\ref{large time}) implies
%$$
%T > \frac{(\mu_0+2Rm)^2}{\beta}> \frac{2\mu_0 2Rm}{\beta}
%= \frac{4R\mu_0}{\beta}m.
%                                           $$
%That is, the observation time $T$ must be at least proportional to $m$.
%This fact suggests that 
%if $\arg H(t)$ changes very rapidly for $0\le t \le T$, then 
%$m$ may be arbitrarily large independently of $T>0$, so that (32) cannot
%hold. \\
%\textcolor{red}{Put here Picture 4}
%\begin{cex}\rm
% For simplicity, we set $H(t) = (\cos \tau_0 t, \sin \tau_0 t)$
%with some constant $\tau_0>0$.
%Then, in Lemma ?.?, as is directly seen, we have 
%$$
%m \approx T\left( \frac{4}{\pi}\frac{1}{\tau_0}\right)^{-1}
%= \frac{\pi}{4}\tau_0T.
%$$
%Hence (32) requires
%$$
%T > \frac{4R\mu_1}{\beta}\left(\frac{4}{\pi}\tau_0T-1\right).
%$$
%If $\tau_0>0$ is sufficiently large, that is, $H(t)$ changes angles 
%rapidly,
%then this suggests that condition (\ref{large time}) is impossible.
%%\\
%???
%\end{cex}
%}
%\vspace{0.5cm}

%\subsubsection*{Some comments}
We conclude this introduction with some comments on our main results.
\begin{enumerate}

\item One could establish an estimate similar to the one in Theorem \ref{thm obs} with the 
maximum norm by the method of characteristics.  Our proof is based on  
Carleman estimates, which naturally provide $L^2$-estimates for
solutions over $\Omega \times \{ t\}$.  The method of cha\-racteristics does not
yield such global $L^2$-estimates directly.
$L^2$-estimates, not estimates in the maximum norm, 
are related to  exact controllability and are more flexi\-bly applied to other 
pro\-blems such as inverse problems, although we  discuss no such aspects in this paper.

\item Although, due to the simplicity of equation (\ref{E1}), the method of characteristics can be easily applied to explain the validity of observability results, the one point we would like to stress is the fact that, in this paper, we intend to derive a Carleman estimate under minimal assumptions. Essentially, we want to give an explicit construction of the weight function that only depends on the lower bound (\ref{H1}) and the modulus of continuity of $H.$

\item It is worth noting that Theorem~\ref{thm obs}  aims at the determination of
the solution $u$ on the whole cylinder  $\Omega \times [0,T]$, not only of $u(\cdot,0)$ in $\Omega$.
For this reason, in Theorem \ref{thm obs}, we have to measure data on the whole lateral boundary $\partial\Omega \times
(0,T)$, not just on a subboundary as we did for the Carleman estimate in Theorem~\ref{Carleman}---where, however, 
the norm of  $u(\cdot,T)$ in $\Omega$ is included. The fact that measurements on the whole boundary are necessary  to majorize $u$ on  $\Omega \times [0,T]$ can be easily understood by looking at the representation solutions given by characteristics.

\item Another purpose of this paper is to single out  an  assumption which suffices to derive observability from a Carleman estimate. We do so with condition (\ref{max cond}), which has a clear geometric meaning: one requires $H(t)$ not to oscillate too much for enough time, giving an explicit evaluation of such a time. We do not pretend our method to provide the optimal evaluation of the  observabi\-lity time. %, see Example ?? . 
On the other hand, Example \ref{cex} shows that some assumption is needed for obser\-vability: (\ref{max cond}) is an example of a sufficient quantitative condition for the  observability of solutions on $\Omega \times [0,T]$.
\end{enumerate}

%We can establish an estimate similar to Theorem \ref{thm obs} with the 
%maximum norm by the method of characteristics.  Our proof is based on the 
%Carleman estimate, which naturally provides an $L^2$-estimate of 
%solutions over $\Omega \times \{ t\}$.  The method of cha\-racteristics does not
%yield such a global $L^2$-estimates directly.
%The $L^2$-estimate, not an estimate in the maximum norm, 
%is related to the exact controllability and more flexi\-bly applied to other 
%pro\-blems such as inverse problems, although we do not here discuss details.\\
%%Namely, 
%Although due to the simplicity of equation (\ref{E1}) the method of characteristics can be easily applied to explain the observability result as
%one can expect for our problem, the one point we would like to stress is the fact that, in this paper, we intend to derive a Carleman estimate under minimal assumptions. Essentially, we wanted to give an explicit construction of the weight function that only depends on the lower bound (\ref{H1}) and the modulus of continuity of $H.$\\

\subsubsection*{Main references and  outline of the paper}
Carleman estimates for transport equations are proved in 
Gaitan and Ouzzane \cite{GO}, G\"olgeleyen and Yamamoto \cite{GY},
Klibanov and Pamyatnykh \cite{KP}, 
Machida and Yamamoto \cite{MM} to be applied to inverse problems of 
determining spatially varying coefficients, where coefficients of
the first-order terms in $x$ are assumed not to depend on $t$.
%The results in \cite{GO}, \cite{GY} and \cite{MM}
%are $L^2$-stability estimates in $\Omega$ for inverse problems.
In order to improve results for inverse problems 
by the application of Carleman estimates, we need 
a better choice of the weight function in the Carleman estimate.
The works \cite{GO} and \cite{KP} use one weight function which is 
very conventional
for a second-order hyperbolic equation %{\color{red} 
but seems less useful to derive 
analogous results  for a time-dependent function $H(t)$.  Our choice is more similar 
to the one in \cite{MM} and \cite{GY}, but even these
papers allow no time dependence for $H$.
Although it is very difficult to choose the best possible weight function 
for the partial differential equation under consideration,
our choice (\ref{wf}) of the weight function seems more adapted for the 
nature of the transport equation (\ref{E1}).

As is commented above, the method of characteristics is applicable to 
inverse problems for first-order hyperbolic systems as well as 
 transport equations and we refer for example to 
Belinskij \cite{Be} and Chapter 5 in Romanov \cite{R}, which discuss an 
inverse problem of determining an $N\times N$-matrix $C(x)$ in 
$$
\partial_tU(x,t) + \Lambda\partial_xU(x,t) + C(x)U(x,t) = F(x,t),
\quad 0<x<\ell, \, t>0
$$
with a suitably given matrix $\Lambda$ and vector-valued function $F$.
The works \cite{Be} and \cite{R} apply the method of 
characteristics to prove the uniqueness and the existence 
of $C(x)$ realizing extra data of $U$ provided that $\ell>0$ is 
sufficiently small.

The method by Carleman estimates for establishing both energy estimates 
like Theorem 1.6 and inverse problems of determining 
spatial varying functions is well-known for hyperbolic and 
parabolic equations and we refer to 
Beilina and Klibanov \cite{BK}, Bellassoued and Yamamoto \cite{BY},
Yamamoto \cite{Y}.  
\\

The plan of the paper is the following. 
In Section \ref{SecCarleman}, we prove the Carleman estimate (Theorem \ref{Carleman}). In Section \ref{Sec obs}, we obtain the observability inequality (Theorem \ref{thm obs}). 
%In Section \ref{threshold time} there is the explanation about the threshold time $T_0$ to obtain the observability. 
Finally, in the Appendix we put the proof of Lemma \ref{Pr1}.

\section{Proof of the Carleman estimate}\label{SecCarleman}
%Without loss of generality let us suppose that ${\bf 0}=(0,\ldots,0)\in\overline{\Omega}$.
%Otherwise we can translate $\Omega$ suitably to reduce to the case of
%${\bf 0} \in \ooo{\OOO}$.  
Let $S_*\in\left(\frac{1}{\sqrt{2}},1\right)$ and %Then, by Lemma \ref{Pr1}, we can choose 
$\{t_j\}_0^m$ a partition 
of $[0,T]$ associated to $H(t)$ such that (\ref{H2}) is satisfied. 
%Let us set $\displaystyle \eta_j:=\frac{H(t_j)}{|H(t_j)|},\;j=0\ldots,m-1,$ 
%and $r_{\Omega}:=diam(\Omega)=\displaystyle \sup_{x,y\in\overline{\Omega}}|x-y|$.
\subsection{Some preliminary lemmas %$\varphi$
}\label{sec wf}
%We define the weight function $\va(x,t)$ on $%Q = 
%\ooo{\OOO}\times [0,T)$ %as the continuous function in the points $(x,T)$ 
%as follows 
%%\begin{equation}\label{wf}
%$$
%\va(x,t) = \va_j(x,t) := -\beta t + \vert x-\xxxj\vert^2, \;
%x\in%\Omega
%\ooo{\OOO},\;t\in [t_{j}, t_{j+1}), \; j=0,1,..., m-1,% \eqno{(1.8)}
%$$%\end{equation}
%where $\beta\in\R$ will be chosen later,
%% $x_j=-R_j\eta_j$ with 
%%some $R_j>0$ chosen so that $x_j\not\in\overline{\OOO}$ and
%\begin{equation}\label{cone}
%\angle C_{min}(\xxxj) \le\frac{\pi}{4}.
%\end{equation}
%Here $C_{min}(\xxxj)$ denotes the minimum cone 
%with the apex $\xxxj$ which includes $\OOO$ and $\angle C_{min}(\xxxj)$ 
%denotes the apex angle of the cone $C_{min}(\xxxj)$.\\
%We remark that condition (\ref{cone}) is equivalent to
%$$
%(x-x_j)\cdot\eta_j\ge\frac{\sqrt{2}}{2}|x-x_j|,\;\; \forall x\in \overline{\OOO},
%$$
%that is,
%\begin{equation}\label{wf0}
%(x+R_j\eta_j)\cdot\eta_j\ge\frac{\sqrt{2}}{2}|x+R_j\eta_j|,\;\; \forall 
%x\in \overline{\OOO}.
%\end{equation}
%Such $x_j=-R_j\eta_j$ exists if we choose $\vert \xxxj\vert=R_j$ 
%sufficiently large 
%compared with $\delta_{\OOO}$.
%In the following lemma we give a lower estimate for $R_j.$ 
%\begin{lem}\label{Prx0j}
%If $R_j\ge (3+2 \sqrt{2})\delta_{\OOO}$, 
%then $x_j=-R_j\eta_j$ satisfies the condition (\ref{wf0})\,.
%\end{lem}
\begin{lem}\label{Prx0j}
Given $R_j,\,j=0,\ldots,m-1,$ as in %\ge (3+2 \sqrt{2})\delta_{\OOO}
 %as in
  (\ref{Rj}), 
then %$x_j=-R_j\eta_j$ satisfies the condition 
\begin{equation}\label{wf0}
(x+R_j\eta_j)\cdot\eta_j\ge S_*%\frac{\sqrt{2}}{2}
|x+R_j\eta_j|,\;\; \forall 
x\in \overline{\OOO},
\end{equation}
where $\eta_j$ are defined in (\ref{etaj}).
%(\ref{wf0})\,.
\end{lem}

%The proof of Lemma \ref{wf0} is given in Section \ref{sec wf}.\\

\begin{proof} % (of Lemma \ref{Prx0j}).
For every $x\in \overline{\OOO}$, we have 
$|x|=|x-{\bf 0}|\le \delta_{\OOO}$ since ${\bf 0}\in\overline{\OOO}$, and 
\begin{equation}\label{wf1}
S_*
%\frac{\sqrt{2}}{2} 
|x+R_j\eta_j|\le %\frac{\sqrt{2}}{2}
S_*\left(|x|+R_j|\eta_j|\right)=S_*%\frac{\sqrt{2}}{2}
\left(|x|+R_j \right)\le S_*%\frac{\sqrt{2}}{2}
\left( 
\delta_{\OOO}+R_j\right),
\end{equation}
and, since $-x\cdot\eta_j\le |x\cdot\eta_j|\le |x||\eta_j|=|x|
\le \delta_{\OOO},$
\begin{equation}\label{wf2}
(x+R_j\eta_j)\cdot\eta_j=x\cdot\eta_j+R_j \eta_j\cdot\eta_j
=x\cdot\eta_j+R_j\ge R_j-|x|\ge R_j-\delta_{\OOO}.
\end{equation}
From (\ref{wf1}) and (\ref{wf2}) it follows that a sufficient condition for the inequality (\ref{wf0}) is the following
$$ 
R_j-\delta_{\OOO}\ge %\frac{\sqrt{2}}{2}
S_*( \delta_{\OOO}+R_j),
$$
that is,
$R_j\ge \frac{1+S_*}{1-S_*}%(3+2 \sqrt{2})
 \delta_{\OOO}.$ For every $j=1,\ldots,m-1,$ the last condition is verified by $R_j$ defined as in (\ref{Rj}).
\smartqed
\qed
\end{proof}

%\begin{rem}\label{x0j}
%\rm We note that if ${\bf 0}\not\in\Omega,$ chosen $x^*\in\Omega$ as a landmark% point and let us set $\delta^*:=\delta_{\OOO}+|x^*|,$
%  if $R_j\ge (3+2 \sqrt{2}) \delta^*$ thus $x_j=-R_j\eta_j$ satisfies the con%dition (\ref{wf0})\,. This is consequence of the proof of the Proposition \ref{%Prx0j} since 
%  $$|x|\le |x-x^*|+|x^*|\le \delta_{\OOO}+|x^*|= \delta^*, \;\forall x\in\OOO.
%$$
%  The best $\delta^*$ is equal to $ \delta_{\OOO}+dist({\bf 0},\Omega).$
%\end{rem}
%
%In addition to the condition (\ref{wf0}), assumed
%by $R_j\ge (3+2 \sqrt{2})\delta_{\OOO}$, we choose 
%$x_j,\;j=0,\ldots,m-1$ such that
%\begin{equation}\label{wf3}
%M_{j} < \mu_{j+1}, \quad j=0,1, ..., m-2,   
%\end{equation}
%where 
%\begin{equation}\label{wf4}
%M_j := \max_{x\in \ooo{\OOO}} \vert x-\xxxj\vert\quad \mbox{ and } 
%\quad\mu_{j} := \min_{x\in \ooo{\OOO}} \vert x-x_0^{j}\vert,
%\quad j=0,\ldots,m-1.
%\end{equation}
%%For the following
%Thereafter, let us set
%$$
%\mu_0 := \min_{j=0,\ldots,m-1} \mu_{j} \qquad\mbox{ and }\qquad   M_*
%:=\max_{j=0,\ldots,m-1} M_j.
%$$
By the definition (\ref{Rj}) of the sequence $\{R_j\}$ the following Lemma \ref{M<mj} follows.
\begin{lem}\label{M<mj}
Let $\; x_j=-R_j\eta_j,\;\;j=0,\ldots,m-1,$ with $R_j$ defined as in (\ref{Rj}). Then
\begin{equation}\label{wf3}
M_\Omega(x_j)=\max_{x\in\overline{\Omega}}|x-x_j|<
%=:
\min_{x\in\overline{\Omega}}|x-x_{j+1}|=d_\Omega(x_{j+1})
%\mu_{j+1}
,\qquad j=0,\ldots,m-2.
\end{equation}
\end{lem}
%\begin{proof}
%
%\smartqed
%\qed
%\end{proof}
%Then, we can choose $\beta\in(0,\sqrt{2}\,H_0\mu_0)$.\\
%\begin{rem}
%\rm We note that the points $x_j, \;j=0,\ldots,m-1,$ are further away from $\bf 0$ as $j$ increases.
%\end{rem}
%A consequence of 
By Lemma \ref{M<mj} (see also Figure \ref{points})
%is that the choice of the $R_j$'s guarantees that the points $x_j$'s are located %in a position by 
%to distance in\-creasing, respect to $j$, from the origin.
%So, 
we deduce
\begin{equation}\label{Maxmin}
\max_{j=0,\ldots,m-1}M_\Omega(x_j)=M_\Omega(x_{m-1})%M_*%M_{m-1}
\qquad \mbox{ and }\qquad \min_{j=0,\ldots,m-1}d_\Omega(x_j)=d_\Omega(x_0).%\mu_0,
 \end{equation}

\begin{lem}\label{Pr Hdot}
Let $\; x_j=-R_j\eta_j,\;\;j=0,\ldots,m-1,$ with $R_j$ defined as in (\ref{Rj}).
%Given $R_j,\,j=1,\ldots,m-1,$ as in %\ge (3+2 \sqrt{2})\delta_{\OOO}
% %as in
%  (\ref{Rj}), 
%
%If the points $x_j=-R_j\eta_j,\;j=0,1,\ldots,m-1,$ satisfy condition (\ref{wf0}), 
%
Then,
$$
H(t)\cdot (x-\xxxj) \ge %\frac{\sqrt{2}}{2}
%\arccos(
%(2S_*^2-1)
C_*\,H_0\,%\mu_0
d_\Omega(x_0),
%>(\sqrt{2}-1)H_0\,%\mu_0
%d_\Omega(x_0), 
\quad
t_{j} \le t \le t_{j+1}, \thinspace j=0,\ldots, m-1, \;\; x\in \ooo{\OOO},   
$$
where $C_*=2S_*^2-1>0$ and $\displaystyle\,H_0=\min_{t\in[0,T]}|H(t)|>0.$ 
\end{lem}
\begin{proof}
Let $\vartheta^*\in(0,\pi/4)$ be such that $\cos\vartheta^*=S_*.$ For $t\in[t_{j}, t_{j+1}],\, j=0,\ldots, m-1$, from (\ref{wf0}) and Remark~\ref{rmk:H2} we deduce that
$$
%2
H(t)\cdot (x-\xxxj) 
%=2\vert H(t)\vert \vert x-\xxxj\vert \cos\angle (x-\xxxj, 
%H(t))%\ge \sqrt{2}\,H_0m_{j}
\ge %\sqrt{2}
%\arccos(
\cos2\vartheta^*%2S_*
%)
\,H_0 d_\Omega(x_j)\ge %sqrt{2}
(2S_*^2-1)
\,H_0 d_\Omega(x_0)
%\mu_0
, \quad x\in \ooo{\OOO}
$$
which is our conclusion.
\smartqed
\qed
\end{proof}

%where, for simplicity of notation, %let us set
%%\begin{equation}\label{M*}
%$M_*:=M_\Omega(x_{m-1}).$\\
%\begin{equation}\label{Hdot}
%%\mu_0 > \sqrt{2}\,H_0T\quad \mbox{ and }\quad
%\beta:= (\sqrt{2}-1)\,H_0\mu_0. %\quad 
%%\frac{T}{m}.
%\end{equation}
%In particular, (\ref{Hdot}) yields
%\begin{equation}\label{17}
%\mu_j^2 - \beta t_j \ge \mu_0^2 - \sqrt{2}\,H_0\mu_0T%\frac{T}{m}
%> 0.
%\end{equation}

\subsection{Derivation of the Carleman estimate}

After introducing the previous lemmas in Section 
\ref{sec wf}, we are able to prove Theorem \ref{Carleman}. In this section, for simplicity of notation, for $j=0,\ldots,m-1$ let us set 
\begin{equation}\label{aster}
M_j:=M_\Omega(x_j)\qquad\mbox{ and }\qquad
\mu_j:=d_\Omega(x_j),
\end{equation}
 see  (\ref{wf4}) for the definitions of $M_\Omega(x_j)$ and $d_\Omega(x_j)$.
\begin{proof}{(of Theorem \ref{Carleman}).} 
We derive a Carleman estimate on $$Q_j :=\OOO \times (t_j,t_{j+1}),\;\; 0\le j 
\le m-1.$$  Let $w_j := e^{s\va_j}u,$ where $\va_j$ is defined in (\ref{wf}), 
and
\begin{equation}\label{Lwj1}
L_jw_j := e^{s\va_j}P(e^{-s\va_j}w_j).
\end{equation}
By direct calculations, we obtain
\begin{equation}\label{Lwj2}
L_jw_j = \ppp_tw_j + H(t)\cdot \nabla w_j - s(P\varphi_j)w_j
\quad \mbox{in}\;\; \;Q_j,
\end{equation}
where, keeping in mind (\ref{wf}) and the definition of the operator $P$ 
contained in (\ref{E1}),
$$ 
P\varphi_j(x,t)= \ppp_t\va_j + %\alpha'
H(t)\cdot \nabla \va_j 
= -\beta + 2%\alpha'
H(t)\cdot (x-\xxxj), \quad 
0 \le j \le m-1.
$$
By Lemma \ref{Pr Hdot} and (\ref{Hdot}), since $\beta=(2S_*^2-1)H_0\mu_0\in\big(0, 2(2S_*^2-1)H_0\mu_0\big)$ we have
\begin{equation}\label{fj1}
P\varphi_j=-\beta + 2H(t)\cdot (x-\xxxj)\ge \,C_*H_0\mu_0,
\end{equation}
where $C_*=2S_*^2-1$.
Therefore, by (\ref{fj1}) we obtain
\begin{eqnarray}\label{Op}
\int_{Q_j}\vert L_jw_j\vert^2 dxdt
&\ge&  - 2s\int_{Q_j} (P\varphi_j)w_{j}(\ppp_tw_j + H(t)\cdot \nabla w_j) dxdt
\nonumber\\
&+& 
s^2\int_{Q_j} \vert 2H(t)\cdot (x-\xxxj) - \beta\vert^2 |w_j|^2 dxdt
\nonumber\\
&\ge& 
I_1+I_2+C_*^2H^2_0\mu_0^2\,s^2 \int_{Q_j}|w_j|^2 dxdt,
\end{eqnarray}
where
$$
I_1 :=  - 2s\int_{Q_j} (P\varphi_j)w_j\ppp_tw_j dxdt\quad \mbox{ and } \quad  
I_2 : = -2s\int_{Q_j} (P\varphi_j)H(t)\cdot (w_j\nabla w_j) dxdt.
$$
We have
\begin{eqnarray}\label{fj2}
I_1 &=&  - 2s\int_{Q_j} (P\varphi_j)w_j\ppp_tw_j dxdt
= - s\int^{t_{j+1}}_{t_{j}}\int_{\OOO} (P\varphi_j)
\ppp_t(w_j^2) dxdt\nonumber\\
&=& s\int_{\OOO} \left[P\varphi_j(x,t)|w_j(x,t)|^2\right]
^{t=t_{j}}_{t=t_{j+1}} dx
+ s\int_{Q_j} \ppp_t(P\varphi_j(x,t))|w_j|^2 dxdt.
\end{eqnarray}
Recalling (\ref{aster}), we
obtain 
$$
\ppp_t(P\varphi_j(x,t)) = 2(x-x_j)\cdot H^\prime(t)
%= 2\vert x-x_j\vert \vert H'(t)\vert \cos\angle (x-x_j, H'(t))
%$$
%$$
\ge -2M_{m-1}\max_{t\in[0,T]}|H'(t)|=:-H_0'.
$$ 
Consequently, from (\ref{fj2}) we deduce
\begin{equation}\label{fj3}
I_1\ge s\int_{\OOO} \left[P\varphi_j(x,t)|w_j(x,t)|^2\right]
^{t=t_{j}}_{t=t_{j+1}} dx
- s\,H_0'\int_{Q_j} |w_j|^2 dxdt.
\end{equation}
Then, for $I_2$ we deduce 
\begin{eqnarray}
 I_2  &=& -2s\int_{Q_j} (P\varphi_j)H(t)\cdot (w_j\nabla w_j) dxdt
= -s\int^{t_{j+1}}_{t_{j}} \int_{\OOO}P\varphi_j
\sum_{k=1}^d H_k(t)\ppp_k(w_j^2) dxdt\nonumber\\
&=&  s\int^{t_{j+1}}_{t_{j}} \int_{\OOO}
 \sum_{k=1}^d (\ppp_k(P\varphi_j))H_k(t)|w_j|^2 dxdt
- s\int^{t_{j+1}}_{t_{j}} \int_{\ppp\OOO} P\varphi_j(H(t)\cdot\nu(x))
\vert w_j\vert^2 d\gamma dt. \nonumber
\end{eqnarray}
We note that   
\begin{equation}\label{H*}
H(t)\cdot (x-\xxxj)\le|H(t)||x-\xxxj|\le H_*M_*,
\end{equation}
where we set (see (\ref{Maxmin})) $$ M_*=M_{m-1}
%\!\!\!\!\!\max_{j=0,\ldots,m-1} 
%\!\!\max_{x\in \ooo{\OOO}} \vert x-x_0^{j}\vert>0,
\qquad\mbox{ and }\qquad
%\displaystyle 
H_*:=\!\!\max_{t\in[0,T]}|H(t)|>0.$$
Therefore, since $P\varphi_j > 0$ by (\ref{fj1}) and $\ppp_k(P\varphi_j)=2H_k(t)$,
we estimate $I_2$ in the following way: 
\begin{eqnarray}\label{I2}
I_2  
&=&  2s\int^{t_{j+1}}_{t_{j}} \int_{\OOO}%\int_{Q_j}
 \sum_{k=1}^d  H^2_k(t)|w_j|^2 dxdt
- s\int^{t_{j+1}}_{t_{j}} \int_{\ppp\OOO} P\varphi_j(H(t)\cdot\nu(x))
\vert w_j\vert^2 d\gamma dt\nonumber\\
&\ge&  2s\int^{t_{j+1}}_{t_{j}} \int_{\OOO}
 |H(t)|^2|w_j|^2 dxdt \nonumber\\
&-& s \int_{\Sigma_j} (-\beta + 2H(t)\cdot (x-\xxxj))(H(t)\cdot\nu(x))
\vert w_j\vert^2 d\gamma dt\nonumber\\
&\ge&  2\,H_0^2s\int^{t_{j+1}}_{t_{j}} \int_{\OOO} \vert w_j\vert^2
dxdt 
- 2s\int_{\Sigma_j} (H(t)\cdot (x-\xxxj))(H(t)\cdot\nu(x))
\vert w_j\vert^2 d\gamma dt\nonumber\\
&\ge& 2\,H_0^2s\int_{Q_j} \vert w_j\vert^2 dxdt
- 2H_*M_*s \int_{\Sigma_j}|H(t)||\nu(x)| \vert w_j\vert^2 d\gamma dt\nonumber\\
&\ge& 2\,H_0^2s\int_{Q_j} \vert w_j\vert^2 dxdt
- 2H_*^2M_*s \int_{\Sigma_j} \vert w_j\vert^2 d\gamma dt,
\end{eqnarray}
where
$$%\begin{equation}\label{5}
\Sigma_j = \{ (x,t) \in \ppp\OOO \times(t_j,t_{j+1}):
\thinspace H(t)\cdot \nu(x) \ge 0 \}. 
$$%\end{equation}
Hence, by (\ref{Op}), (\ref{fj3}) and (\ref{I2}), we obtain
\begin{eqnarray}%\label{J0}
 \int_{Q_j}\vert L_jw_j\vert^2 dxdt
 %\int^{t_{j+1}}_{t_{j}}\int_{\OOO} &\vert& Pu\vert^2 e^{2s\va_j} dxdt
& \ge & s\int_{\OOO} \left[P\varphi_j(x,t)|w_j(x,t)|^2\right]
^{t=t_{j}}_{t=t_{j+1}} dx
\nonumber\\
&-& H_0' s\,%(\sqrt{2} \mu_0\,H_0^\prime+2\,H_0^2)
\int_{Q_j} |w_j|^2 dxdt 
+ %H^2_0 \mu_0^2\,
C_1 s^2 \int_{Q_j}|w_j|^2 dxdt   \nonumber\\
&-& 2H_*^2M_*s \int_{\Sigma_j} \vert w_j\vert^2 d\gamma dt\,,\nonumber
\end{eqnarray}
for some positive constant $C_1.$
Since $w_j := e^{s\va_j}u,$ from the previous inequality, for $j=0,\ldots,m-1$,
by (\ref{Lwj1}) we deduce that there exists also a positive constant $C_2$ such that
\begin{eqnarray}\label{J0}
\int^{t_{j+1}}_{t_{j}}\int_{\OOO} \vert Pu\vert^2 e^{2s\va_j} dxdt 
 &\ge& s\int_{\OOO}\psi_j(x)dx+(C_1s^2-H_0's)\int_{Q_j} e^{2s\va_j}|u|^2
                       dxdt \nonumber\\
&-&C_2s e^{C_2s}
\int_{\Sigma_j} \vert u\vert^2 d\gamma dt, %C_2\,s^2 \int_{Q_j}|w_j|^2 dxdt
\end{eqnarray}
where $C_1,\,C_2$ are positive constants and
$$\psi_j(x):=
%J:=s\int_{\OOO} 
\left[P\varphi_j(x,t) e^{2s\va_j(x,t)}\,|u(x,t)|^2\right]
^{t=t_{j}}_{t=t_{j+1}}.
$$
By (\ref{wf}) and (\ref{fj1}) we obtain
\begin{eqnarray}
\psi_j(x)&=&\left[\left(2H(t)\cdot (x-\xxxj) - \beta\right)
e^{2s(-\beta (t-t_j) + \vert x-\xxxj\vert^2)}\vert u(x,t)\vert^2 %e^{2s\va_j(x,t)}
%\,|u(x,t)|^2 
\right]^{t=t_{j}}_{t=t_{j+1}}\nonumber\\
&=& (2H(t_{j})\cdot (x-\xxxj) - \beta)
e^{2s%-\beta (t-t_{j}) 
\vert x-\xxxj\vert^2}\vert u(x,t_{j})\vert^2\nonumber\\
&-& (2H(t_{j+1})\cdot (x-\xxxj) - \beta)
e^{2s(-\beta (t_{j+1}-t_j) + \vert x-\xxxj\vert^2)}\vert u(x,t_{j+1})\vert^2.
\end{eqnarray}
Therefore, summing in $j$ from 0 to $m-1$ and keeping in mind that 
$t_0=0$ and $t_m=T$ 
%and $u(x,T)=0$,
 by (\ref{Hdot}) and (\ref{H*}) %and (\ref{17}) 
we have
\begin{eqnarray}\label{psij}
\sum_{j=0}^{m-1}\psi_j(x)&\ge&(2H(0%t_{0}
)\cdot (x-x_0) - \beta)
e^{2s(%-\beta t_{0} + 
\vert x-x_0\vert^2)}\vert u(x,%t_{0}
0)\vert^2%\nonumber\\%&+&
+\sum_{j=1}^{m-1}%e^{-2s\beta t_{j}}
\!q_j(x)
\vert u(x,t_{j})\vert^2\nonumber\\
&-& (2H(T)\cdot (x-x_{m-1}) - \beta)
e^{2s(-\beta (T-t_{m-1}) + \vert x-x_{m-1}\vert^2)}\vert u(x,T)\vert^2
\nonumber\\
&\ge&\mu_0H_0e^{2s%-\beta t_{0} + 
\mu_0^2%\vert x-x_0\vert^2
}\vert u(x,%t_{0}
0)\vert^2 -2M_*H_*e^{2sM_*^2}\vert u(x,T)\vert^2\nonumber
+ \sum_{j=1}^{m-1}%e^{-2s\beta t_{j}}
\!q_j(x)
\vert u(x,t_{j})\vert^2,
\end{eqnarray}
where, for $j=1,\ldots,m-1,$ we set
$$q_j(x):=%\biggl[
(2H(t_{j})\cdot (x-\xxxj) - \beta)
e^{2s \vert x-\xxxj\vert^2}-\left(2H(t_{j})\cdot \left(x-x_{j-1}\right) - \beta\right)
e^{2s \left\vert x-x_{j-1}\right\vert^2}
%\biggl]e^{-2s\beta t_{j}}
.$$
Thus, by (\ref{wf4}), (\ref{aster}), (\ref{fj1}) and (\ref{H*}), we obtain the following estimate
\begin{eqnarray}
q_j(x)\ge\tilde{C}\mu_0H_0e^{2s\mu^2_j}-H_*M_*e^{2sM^2_{j-1}}\nonumber
=\tilde{C}\mu_0H_0e^{2s\mu^2_j}\left(1-\frac{M_*H_*}{\tilde{C}\mu_0H_0}e^{-2s\left(\mu^2_j-M^2_{j-1}\right)}\right).
\end{eqnarray}
Thanks to (\ref{wf3}) (see Lemma \ref{M<mj}), the choice of the points $x_j$ permits to have $\mu_j-M_{j-1}>0$, then we deduce that there exist $s_j>0$ enough large, that is $s_j>\frac{1}{2\left(\mu^2_j-M^2_{j-1}\right)}\log\left(\frac{2H_*M_*}{\tilde{C}\mu_0H_0}\right),\;j=1,\ldots,m-1,$ such that, for every $\displaystyle s>s_0:=%\left\{
\max_{j=1,\ldots,m-1}s_j%,\right\}
,$ we have
%$$
\begin{equation}\label{qj} 
%e^{-2s\beta t_{j}}
q_j(x)\ge 
%e^{-2s\beta t_{j}
%} 
\frac{\mu_0H_0}{2}e^{2s\mu_j^2}\ge 
%e^{-2s\beta T%t_{j}
%} 
\frac{\mu_0H_0}{2}e^{2s\mu_0^2}\ge C_0e^{C_0s},
\end{equation}
for some positive constant $C_0=C_0(s).$ Thus, by (\ref{J0}), (\ref{psij}), and (\ref{qj}) we have that
%with $C_0=C_0(s)=\min\{
% \frac{\mu_0H_0}{
%2}e^{-2s\beta T%t_{j}
%},2m^2_0\}.$
\begin{eqnarray}\label{J1}
% \int_{Q_j}\vert L_jw_j\vert^2 dxdt
\int_{Q} \vert Pu\vert^2 e^{2s\va} dxdt &=&\sum_{j=0}^{m-1} \int^{t_{j+1}}_{t_{j}}\int_{\OOO} \vert Pu\vert^2 e^{2s\va_j} dxdt   \nonumber\\
& \ge& s\sum_{j=0}^{m-1}\int_{\OOO}\psi_j(x)dx+(C_1s^2-H_0's)\sum_{j=0}^{m-1}\int_{Q_j} e^{2s\va_j}  |u|^2 dxdt \nonumber\\
&-&C_2se^{C_2s} \sum_{j=0}^{m-1}
\int_{\Sigma_j} \vert u\vert^2 d\gamma dt %C_2\,s^2 \int_{Q_j}|w_j|^2 dxdt
\nonumber\\
& \ge& C_3{s}^2%\sum_{j=0}^{m-1}
\int_{Q} e^{2s\va_j} 
|u|^2 dxdt -C_2se^{C_2s} \sum_{j=0}^{m-1}
\int_{\Sigma_j} \vert u\vert^2 d\gamma dt
\nonumber\\ &+&C_0se^{C_0s}\sum_{j=0}^{m-1}\int_{\OOO}\vert u(x,t_{j})\vert^2
%\psi_j(x)
dx -%H_*M_*
 C_2se^{%2
 C_2 s
 %M_*^2
 }\int_{\OOO}\vert u(x,T)\vert^2\,dx\nonumber %C_2\,s^2 \int_{Q_j}|w_j|^2 dxdt
\end{eqnarray}
for any $0<C_3<C_1$ and all $s$ sufficiently large.
The last estimate completes the proof of Theorem \ref{Carleman}.
\smartqed
\qed
\end{proof}

\section{Proof of the observability inequality}\label{Sec obs}
Let us give in Section \ref{energysec} two lemmas and in Section \ref{obssec} the proof of Theorem \ref{thm obs}. 
%and in Section \ref{threshold time} some remarks about the presence of the threshold time to obtain the observability.
\subsection{Energy estimates}\label{energysec}
%The proof of Theorem \ref{thm obs} is obtained by the Carleman estimate obtained in Theorem \ref{Carleman} and a usual cut-off function argument. 
%Before the proof of Theorem \ref{thm obs} 
Let us give %it is need 
the following energy estimates.
\begin{lem}\label{lem energy}
Let $g\in L^2(\ppp\OOO\times (0,T))$ and let us consider the problem
$$%\begin{equation}\label{pb obs bis}
\left\{ \begin{array}{rl}
& \ppp_tu + H(t)\cdot \nabla u = 0 \quad 
\mbox{in $Q:= \OOO\times (0,T)$},\\
& u\vert_{\ppp\OOO \times (0,T)} = g.
\end{array}\right.
\qquad\qquad (\ref{pb obs})$$%\end{equation}
Then, for every $t\in[0,T],$
% exists $T_0>0$ such that if $T>T_0,$
%%{\bf Theorem 3.1.}\\
%%{\it
%%then there exists a 
%for some constant $C>0,$ 
the following energy estimates hold %such that 
\begin{equation}\label{en1}
\Vert u(\cdot,t)\Vert^2_{L^2(\OOO)} \le 
\Vert u(\cdot,0)\Vert^2_{L^2(\OOO)}+H_*\Vert g\Vert^2_{L^2(\ppp\OOO\times (0,T))},
\end{equation}%\nonumber\\ 
\begin{equation}\label{en2}
\Vert u(\cdot,0)\Vert^2_{L^2(\OOO)} \le 
\Vert u(\cdot,t)\Vert^2_{L^2(\OOO)}+H_*\Vert g\Vert^2_{L^2(\ppp\OOO\times (0,T))}, 
\end{equation}%\quad 0\le t \le T,
for any $u \in H^1(Q)$ satisfying (\ref{pb obs}), where $\displaystyle H_*:=\!\!\max_{\xi\in[0,T]}|H(\xi)|.$
%with $g\in 
%L^2(\ppp\OOO\times (0,T))$.
\end{lem}

\begin{proof}
Let $H(t)=(H_1(t),\ldots,H_d(t)),\;t\in [0,T].$
Multiplying the equation in (\ref{pb obs}) by $2u$ and integrating over $\OOO$, we have
$$
\int_{\OOO} 2u\ppp_tu dx + \int_{\OOO}\sum_{k=1}^d H_k(t)2u\ppp_ku dx 
= 0,
$$
then,
$$
\ppp_t\left(\int_{\OOO} \vert u(x,t)\vert^2 dx\right) 
+ \sum_{k=1}^d \int_{\OOO} H_k(t)\ppp_k(\vert u(x,t)\vert^2) dx = 0.
$$
So, integrating by parts, for every $t\in[0,T],$ we obtain
\begin{equation}\label{parts}
\ppp_t\left(\int_{\OOO} \vert u(x,t)\vert^2 dx \right)
=-\sum_{k=1}^d \int_{\ppp\OOO} H_k\vert u\vert^2 \nu_kd\gamma= -\int_{\ppp\OOO} (H\cdot \nu) \vert g\vert^2 d\gamma,
\end{equation}
where $\nu=(\nu_1,\ldots,\nu_d)$ is the unit normal vector outward to the boundary $\ppp\OOO$.
Setting 
$$E(t) := \int_{\OOO} \vert u(x,t)\vert^2 dx,\qquad t\in[0,T],$$ 
by (\ref{parts}), integrating on $[0,t]$  we deduce
$$
\left|E(t) - E(0)\right| = \left|-\int^t_0\int_{\ppp\OOO} 
(H(\xi)\cdot \nu(x)) \vert g(x,\xi)\vert^2 d\gamma d\xi\right|\le H_*\Vert g\Vert_{L^2(\ppp\OOO\times (0,T))}^2
$$
%for $0 \le t \le T$, 
where $\displaystyle H_*=\!\!\max_{\xi\in[0,T]}|H(\xi)|$.
%Therefore there exist constants $C_1, C_2 > 0$ such that 
Thus, for all $t\in[0,T],$ we have
%\begin{equation}\label{en1}
$$
E(t) \le E(0) + H_* \Vert g\Vert_{L^2(\ppp\OOO\times (0,T))}^2,
$$%\end{equation}
and
$$%\begin{equation}\label{en2}
E(0) \le E(t) + H_*\Vert g\Vert_{L^2(\ppp\OOO\times (0,T))}^2.
$$%\end{equation}
%Hence, by (\ref{en2}) we deduce
%\begin{eqnarray}\label{en3}
%\int^{\ep}_0 \int_{\OOO} \vert u\vert^2 dxdt 
%&=& \int^{\ep}_0 E(t) dt
%\ge \int^{\ep}_0 (E(0) - H_*\Vert g\Vert_{L^2(\ppp\OOO\times (0,T))}^2) dt\nonumber\\
%&=& \ep \left(E(0) - H_*\Vert g\Vert_{L^2(\ppp\OOO\times (0,T))}^2\right)
%\end{eqnarray}
%and, by (\ref{en1}) we obtain
%\begin{eqnarray}\label{en4}
%\Vert u\Vert^2_{L^2(Q)} &=& \int^T_0 E(t) dt= \int^T_0 \left(E(0) + H_* \Vert g\Vert_{L^2(\ppp\OOO\times (0,T))}^2\right) dt\nonumber\\
%&\le& E(0)T + H_*T\Vert g\Vert_{L^2(\ppp\OOO\times (0,T))}^2.
%\end{eqnarray}
\smartqed
\qed
\end{proof}
\begin{lem}\label{cor obs}
Let %$T>0,$ 
$0\leq s_1<s_2\leq T,$ $g\in L^2(\ppp\OOO\times (0,T)).$ 
%let us consider the problem (\ref{pb obs}) and %, that is
%\begin{equation}\label{pb obs bis}
%$$\left\{ \begin{array}{rl}
%& \ppp_tu + H(t)\cdot \nabla u = 0 \quad 
%\mbox{in $Q:= \OOO\times (0,T)$},\\
%& u\vert_{\ppp\OOO \times (0,T)} = g.
%\end{array}\right.$$
%\end{equation}
%let
%$0\leq s_1<s_2\leq T.$ %$\;I\subset [0,T]$ %a not empty real interval.
%Let us suppose that for every $t\in [s_1,s_2]$ 
Let us assume that there exists a positive constant $C=C(s_1,s_2)$ such that for every $t\in [s_1,s_2]$ %the subinterval 
%$I$ 
the following observability inequa\-lity holds
\begin{equation}\label{obs sub}
\Vert u(\cdot,t)\Vert_{L^2(\OOO)} \le 
C\Vert g\Vert_{L^2(\ppp\OOO\times (0,T))}, \quad %\mbox{for any} 
\mbox{ for all } u \in H^1(Q)\;  \mbox{ solution to } (\ref{pb obs}). %\forall t\in I, 
\end{equation}
% satisfying (\ref{pb obs}).
%and for some positive constant $C$.\\
% $\displaystyle H_*:=\!\!\max_{\xi\in[0,T]}|H(\xi)|.$
Then, there exists a positive constant $C=C(s_1,s_2,T)$ such that the inequality (\ref{obs sub}) holds for every $t\in[0,T].$
% exists $T_0>0$ such that if $T>T_0,$
%%{\bf Theorem 3.1.}\\
%%{\it
%%then there exists a 
%for some constant $C>0,$ 
%with $g\in 
%L^2(\ppp\OOO\times (0,T))$.
\end{lem}
\begin{proof}
Let $E(t)=\Vert u(\cdot,t)\Vert^2_{L^2(\OOO)},%= \int_{\OOO} \vert u(x,t)\vert^2 dx,
\; t\in[0,T].$ For every $t\in [0,s_1],$ keeping in mind Lemma \ref{lem energy}, by (\ref{en1}), (\ref{en2}) and (\ref{obs sub}) we obtain
\begin{eqnarray}\label{0s_1}
\Vert u(\cdot,t)\Vert^2_{L^2(\OOO)}=E(t) &\le& E(0) + H_* \Vert g\Vert_{L^2(\ppp\OOO\times (0,T))}^2\le E(s_1) + 2H_* \Vert g\Vert_{L^2(\ppp\OOO\times (0,T))}^2\nonumber\\
&\le& (C^2+2H_*) \Vert g\Vert_{L^2(\ppp\OOO\times (0,T))}^2\;.
%C^2\Vert g\Vert^2_{L^2(\ppp\OOO\times (0,T))}$$
\end{eqnarray}
For every $t\in [s_2,T],$ using again Lemma \ref{lem energy}, by (\ref{en1}) and (\ref{obs sub}) we deduce
\begin{equation}\label{s_1s_2}
\!\!\!\!\Vert u(\cdot,t)\Vert^2_{L^2(\OOO)}\!\!=E(t) \le E(s_2) + H_* \Vert g\Vert_{L^2(\ppp\OOO\times (0,T))}^2%\le E(s_2) + 2H_* \Vert g\Vert_{L^2(\ppp\OOO\times (0,T))}^2
\!\!\le (C^2+H_*) \Vert g\Vert_{L^2(\ppp\OOO\times (0,T))}^2\,.
%C^2\Vert g\Vert^2_{L^2(\ppp\OOO\times (0,T))}$$
\end{equation}
From (\ref{0s_1}) %(\ref{obs sub}) 
and (\ref{s_1s_2}) the conclusion follows.
\smartqed
\qed
\end{proof}
\subsection{The proof}\label{obssec}
\begin{proof}{(of Theorem \ref{thm obs}).}\\
Let $\va$ be the weight function given in (\ref{wf}).
By the assumption (\ref{max cond}) it follows that
%Theorem \ref{Carleman}.
%Let 
%$$\max_{0\le j\le m-1}\frac{(t_{j+1}-t_j)d_\Omega(x_j)}{M^2_\Omega(x_j)}>\frac{1}{H_0(\sqrt{2}-1)},$$
%thus 
there exists $j^*\in\{0,\ldots m-1\}$ such that  
\begin{equation}\label{T0}
\frac{(t_{{j^*}+1}-t_{j^*})d_\Omega(x_{j^*})}{M^2_\Omega(x_{j^*})}>\frac{1}{H_0(2S_*^2-1)
%\sqrt{2}-1)
}.
\end{equation}
%Let 
%\begin{equation}\label{T0}
% T >T_0:=\frac{1}{\beta}\displaystyle
% \max_{x\in\ooo{\OOO}} \big\vert x-x_{m-1}\big\vert^2,
%% \frac{{\displaystyle
%% \max_{x\in\ooo{\OOO}} \big\vert x-x_{m-1}\big\vert^2}}{\beta}.
%\end{equation}
By the definition of the weight function $\va(x,t)$  (see (\ref{wf})), it follows that, %on $Q = \OOO\times (0,T)$ 
for every $x \in \ooo{\OOO},$ 
we have  %$ -\beta t + \vert x-\xxxj\vert^2$
%\begin{equation}\label{phi0}
$$\va(x,t_{j^*})=\va_{j^*}(x,t_{j^*}) =%-\beta t_{j^*}+ 
\vert x-x_{j^*}\vert^2
%=\beta(t_{j^*+1}-t_{j^*})-
%\beta t_{j^*+1}+ \vert x-x_{j^*}\vert^2 
> 0
$$
%\end{equation}
%$\va(x,t_{j^*})=\va_0(x,t_{j^*}) =-\beta t_{j^*}+ \vert x-x_0\vert^2 > 0
%$
 and, since (\ref{T0}) holds, keeping in mind that $\beta=(2S_*^2-1)H_0d_\Omega(x_0),
 %(\sqrt{2}-1)H_0\,d_\Omega(x_{0})
 $
% if
%$
% T >T_0:=\frac{1}{\beta}\displaystyle
% \max_{x\in\ooo{\OOO}} \big\vert x-x_{m-1}\big\vert^2,
%% \frac{{\displaystyle
%% \max_{x\in\ooo{\OOO}} \big\vert x-x_{m-1}\big\vert^2}}{\beta}.
%$
%\begin{equation}\label{phi m-1}
$$
%\va(x,t_{j^*+1})=
\lim_{t\rightarrow ({t_{j^*+1}})\!^-}\va_{j^*}(x,%t_{j^*+1}
t)=
%\vert x-x_{j^*}\vert^2 - \beta t_{j^*+1} =
\vert x-x_{j^*}\vert^2 -
\beta(t_{j^*+1}-t_{j^*})%-\beta t_{j^*}
< 0.
$$
%\end{equation}
 %for $x \in \ooo{\OOO}$ 
%$$
% T >\frac{1}{\beta}\displaystyle
% \max_{x\in\ooo{\OOO}} \big\vert x-x_{m-1}\big\vert^2.
%% \frac{{\displaystyle
%% \max_{x\in\ooo{\OOO}} \big\vert x-x_0\big\vert^2}}{\beta}.
%$$
Therefore, there exist $\displaystyle \varepsilon\in\left(
0,%\min\left\{t_1,
\frac{t_{j^*+1}-t_{j^*}}{2}%\right\}
\right)$ and
%$\ep>0$ and 
$\delta > 0$ such that %$\displaystyle \varepsilon<\min\left\{t_1,\frac{T-t_{m-1}}{2}\right\},$
%$0 < \ep < t_1$ and $t_{m-1} < T-2\ep$, 

\begin{equation}\label{phi delta}
\left\{\begin{array}{rl}
\va(x,t)&=\va_{j^*}(x,t) > \delta, \quad\qquad\;\,\quad  \;\;\;\;t\in[t_{j^*},
 %\le t \le
 t_{j^*}+\ep], \thinspace x \in \ooo{\OOO},\\
\va(x,t)&=\va_{j^*}(x,t) < -\delta, \qquad  
t\in [t_{j^*+1}-2\ep, 
%\le t \le 
t_{j^*+1}), \thinspace x \in \ooo{\OOO}.
\end{array}\right.
\end{equation}
%since (\ref{phi0}) and (\ref{phi m-1}) hold and  %by  $0 < \ep < t_1$ and $t_{m-1} < T-2\ep,$ 
%for every $x\in \ooo{\OOO}$ %we see that 
%$$
%\va(x,t) = \va_0(x,t)\;\mbox{ if }\;t\in[0,\ep]%\;\;0 \le t \le \ep \;\;\; 
%\;\;\;\mbox{ and }\;\;\;  \va(x,t) =\va_{m-1}(x,t) \;\mbox{ if }\;t\in[T-2\ep,T].
%%\quad  T-2\ep < t \le T
%$$
%\va(x,t) = 
%\left\{\begin{array}{rl}
%\va_1(x,t), \quad & 0 \le t \le \ep,\;\qquad \quad \,x\in \ooo{\OOO} \\
%\va_{m-1}(x,t), \quad & T-2\ep < t \le T, \quad x \in %\OOO
%\ooo{\OOO},
%\end{array}\right.
%$$
%and so 
%$$\left\{\begin{array}{rl}
%\va_1(x,t) > \delta, \quad & 0 \le t < \ep, \thinspace x \in \ooo{\OOO},\\
%\va_{m-1}(x,t) < -\delta, \quad & T-2\ep < t \le T, \thinspace 
%x \in \ooo{\OOO}.\end{array}\right.               \eqno{(33)}
%$$
Let $u \in H^1(Q),$ satisfying (\ref{pb obs}) on $Q=\Omega\times(0,T)$. Let us consider $Q^*:=\Omega\times(t_{j^*},t_{j^*+1})\subseteq Q.$
% with $g\in 
%L^2(\ppp\OOO\times (0,T))$.
Now we define a cut-off function $\chi \in C^{\infty}_0([t_{j^*},t_{j^*+1}])$ such that 
$0 \le \chi \le 1$ and 
%\begin{equation}\label{(34)}
$$\chi(t) =
\left\{\begin{array}{rl}
1, \quad & t\in[t_{j^*}, t_{j^*+1}-2\ep], \\
0, \quad & t\in[t_{j^*+1}-\ep, t_{j^*+1}].
\end{array}\right.                   %  \eqno{(34)}
$$%\end{equation}
We set 
\begin{equation}\label{v0}
%$$
v(x,t) = \chi(t) u(x,t),\;\;(x,t)\in Q^*,
\end{equation}%$$
then, keeping in mind (\ref{pb obs}) and (\ref{v0}), we deduce
\begin{equation}\label{v}
\left\{ \begin{array}{rl}
& \ppp_tv + H(t)\cdot \nabla v = u(\ppp_t\chi) \quad\quad 
\mbox{in $\;Q^*%:= \OOO\times (0,T)
$},\\
& v\vert_{\ppp\OOO \times (t_{j^*},\,t_{j^*+1})} = \chi g, \\
& v(x,t_{j^*+1}) = 0, \qquad\qquad\quad\;\;\;\; x \in \OOO.
\end{array}\right.
\end{equation}
Applying Theorem \ref{Carleman} to the problem (\ref{v}), since $v(x,t)\leq u(x,t)$ for every $(x,t)\in Q^*$ (see (\ref{v0})), we obtain
\begin{equation}\label{(35)}
s^2\int_{Q^*} \vert v\vert^2 \weight dxdt \le C\int_{Q^*} \vert u\vert^2
\vert \ppp_t\chi\vert^2 \weight dxdt
+ Ce^{Cs}\int_{\Sigma} \vert u\vert^2 d\gamma dt,    
%(\footnote{We note that $%\begin{equation}\label{5}
%\Sigma^*:= \{ (x,t) \in \ppp\OOO \times(t_{j^*},t_{j^*+1}):\,
%\thinspace H(t)\cdot \nu(x) \ge 0 \}. 
%$%\end{equation}
%})
                                          %  \eqno{(35)}
\end{equation}
for all large $s>0$ and 
for some positive constant $C$.  \\  %t_{j^*},t_{j^*+1}
Therefore, by %(33) and 
(\ref{v0}) and (\ref{phi delta}) we have
\begin{equation}\label{o0}
s^2\int_{Q^*} \vert v\vert^2 \weight dxdt 	\ge  s^2\int^{t_{j^*}+\ep}_{t_{j^*}}\int_{\OOO} \vert u\vert^2 e^{2s\va_0} dxdt \ge s^2e^{2s\delta}\int^{t_{j^*}+\ep}_{t_{j^*}}
%\int^{\ep}_0
\int_{\OOO} \vert u\vert^2 dxdt              %\eqno{(36)}
\end{equation}
and, since $\chi \in C^{\infty}_0([t_{j^*},t_{j^*+1}]),$ we also deduce
% the application of (\ref{phi delta})
%(33) and (\ref{(34)}) 
%yields
\begin{eqnarray}\label{o1}
\!\!\!\!\!\!\!\!\!\!\!\!\int_{Q^*} \vert u\vert^2 \vert \ppp_t\chi\vert^2 \weight dxdt
%= \int^{T-\ep}_{T-2\ep} \int_{\OOO} \vert u\vert^2
%\vert \ppp_t\chi\vert^2 \weight dxdt
&=& \int^{{t_{j^*+1}}-\ep}_{{t_{j^*+1}}-2\ep} \int_{\OOO} \vert u\vert^2
\vert \ppp_t\chi\vert^2 e^{2s\va_{m-1}}\! dxdt\nonumber\\
&\le&  K_1e^{-2s\delta}\int^{{t_{j^*+1}}-\ep}_{
{t_{j^*+1}}-2\ep} \int_{\OOO} \vert u\vert^2 dxdt 
\le K_1\Vert u\Vert^2_{L^2({Q^*})}e^{-2s\delta}\!\!\!\!,           
\end{eqnarray}
for all large $s>0$ and 
for some positive constant $K_1$. \\ 
From (\ref{(35)}), by (\ref{o0}) and (\ref{o1}) we obtain
\begin{equation}\label{o2}
s^2e^{2s\delta}\int^{{t_{j^*}}+\ep}_{{t_{j^*}}} \int_{\OOO} \vert u\vert^2 dxdt
\le C_1\Vert u\Vert^2_{L^2({Q^*})}e^{-2s\delta}
+ C_1e^{C_1s}\Vert g\Vert_{L^2(\ppp\OOO\times (0,T))}^2,
                                       % \eqno{(38)}
\end{equation}
for all large $s>0$ and 
for some positive constant $C_1$. \\
%Multiplying (\ref{pb obs}) by $2u$ and integrating over $\OOO$, we have
%$$
%\int_{\OOO} 2u\ppp_tu dx + \int_{\OOO}\sum_{k=1}^d H_k(t)2u\ppp_ku dx 
%= 0,
%$$
%then,
%$$
%\ppp_t\left(\int_{\OOO} \vert u(x,t)\vert^2 dx\right) 
%+ \sum_{k=1}^d \int_{\OOO} H_k(t)\ppp_k(\vert u(x,t)\vert^2) dx = 0.
%$$
%So, integrating by parts, for every $t\in[0,T],$ we obtain
%\begin{equation}\label{parts}
%\ppp_t\left(\int_{\OOO} \vert u(x,t)\vert^2 dx \right)
%=-\sum_{k=1}^d \int_{\ppp\OOO} H_k\vert u\vert^2 \nu_kd\gamma= -\int_{\ppp\OOO} (H\cdot \nu) \vert g\vert^2 d\gamma,
%\end{equation}
%where $\nu=(\nu_1,\ldots,\nu_d)$ is the unit normal vector outward from the boundary $\ppp\OOO$.
%Setting 
%$$E(t) := \int_{\OOO} \vert u(x,t)\vert^2 dx,\qquad t\in[0,T],$$ 
%by (\ref{parts}), integrating on $[0,t]$  we deduce
%$$
%\left|E(t) - E(0)\right| = \left|-\int^t_0\int_{\ppp\OOO} 
%(H(\xi)\cdot \nu(x)) \vert g(x,\xi)\vert^2 d\gamma d\xi\right|\le H_*\Vert g\Vert_{L^2(\ppp\OOO\times (0,T))}^2
%$$
%%for $0 \le t \le T$, 
%where $\displaystyle H_*:=\!\!\max_{\xi\in[0,T]}|H(\xi)|$.
%%Therefore there exist constants $C_1, C_2 > 0$ such that 
%Thus, for all $t\in[0,T],$ we have
%\begin{equation}\label{en1}
%E(t) \le E(0) + H_* \Vert g\Vert_{L^2(\ppp\OOO\times (0,T))}^2,
%\end{equation}
%and
%\begin{equation}\label{en2}
%E(0) \le E(t) + H_*\Vert g\Vert_{L^2(\ppp\OOO\times (0,T))}^2.
%\end{equation}
Setting 
$$E(t) := \int_{\OOO} \vert u(x,t)\vert^2 dx,\qquad t\in[{t_{j^*}},{t_{j^*+1}}],$$ 
 by the energy estimate (\ref{en2}) of Lemma \ref{lem energy} we deduce
\begin{eqnarray}\label{en3}
\int^{{t_{j^*}}+\ep}_{t_{j^*}} \int_{\OOO} \vert u\vert^2 dxdt 
&=& \int^{{t_{j^*}}+\ep}_{t_{j^*}} E(t) dt
\ge \int^{{t_{j^*}}+\ep}_{t_{j^*}} (E({t_{j^*}}) - H_*\Vert g\Vert_{L^2(\ppp\OOO\times (0,T))}^2) dt\nonumber\\
&=& \ep \left(E({t_{j^*}}) - H_*\Vert g\Vert_{L^2(\ppp\OOO\times (0,T))}^2\right)
\end{eqnarray}
and, by the energy estimate (\ref{en1}) of Lemma \ref{lem energy}  we obtain
\begin{eqnarray}\label{en4}
\Vert u\Vert^2_{L^2({Q^*})} &=& \int^{t_{j^*+1}}_{t_{j^*}} E(t) dt= \int^{t_{j^*+1}}_{t_{j^*}} \left(E({t_{j^*}}) + H_* \Vert g\Vert_{L^2(\ppp\OOO\times (0,T))}^2\right) dt\nonumber\\
&\le& E({t_{j^*}})T + H_*T\Vert g\Vert_{L^2(\ppp\OOO\times (0,T))}^2.
\end{eqnarray}
Substituting (\ref{en3}) and (\ref{en4}) into (\ref{o2}), we have
\begin{eqnarray}\label{en5}
s^2e^{2s\delta} \ep  \left(E({t_{j^*}}) - H_*\Vert g\Vert_{L^2(\ppp\OOO\times (0,T))}^2\right)&\le& s^2e^{2s\delta}\int^{{t_{j^*}}+\ep}_{t_{j^*}} \int_{\OOO} \vert u\vert^2 dxdt \nonumber\\
&\le& C_1\Vert u\Vert^2_{L^2({Q^*})}e^{-2s\delta}
+ C_1e^{C_1s}\Vert g\Vert_{L^2(\ppp\OOO\times (0,T))}^2\nonumber\\
&\le& C_1e^{-2s\delta}
\left(E({t_{j^*}})T + H_*T\Vert g\Vert_{L^2(\ppp\OOO\times (0,T))}^2\right)\nonumber\\
&+&C_1e^{C_1s}\Vert g\Vert_{L^2(\ppp\OOO\times (0,T))}^2,\nonumber
%s^2e^{2s\delta} \ep E(0) 
%- s^2e^{2s\delta}H_*\Vert g\Vert_{L^2(\ppp\OOO\times (0,T))}^2
%\nonumber\\
%&\le& C_1E(0)Te^{-2s\delta} 
%+ C_1H_*T\Vert g\Vert_{L^2(\ppp\OOO\times (0,T))}^2e^{-2s\delta}
%+ C_1e^{C_1s}\Vert g\Vert_{L^2(\ppp\OOO\times (0,T))}^2
\end{eqnarray}
for all large $s>0$.  Hence, for all $s$ large enough,
$$
(s^2e^{2s\delta} \ep - C_1Te^{-2s\delta})E({t_{j^*}}) 
\le \left(C_1e^{C_1s}+s^2e^{2s\delta} \ep H_*+C_1e^{-2s\delta}H_* T\right) \Vert g\Vert_{L^2(\ppp\OOO\times (0,T))}^2
$$
 But, for $s>0$ enough large, 
$s^2e^{2s\delta} \ep - C_1Te^{-2s\delta}>0.$
Thus, using again (\ref{en1}), 
for every $t\in[t_{j^*},t_{j^*+1}],$ we obtain
$$\Vert u(\cdot,t)\Vert_{L^2(\OOO)}=E(t)\le E(t_{j^*}) + H_* \Vert g\Vert_{L^2(\ppp\OOO\times (0,T))}^2 \le 
C_2\Vert g\Vert_{L^2(\ppp\OOO\times (0,T))}, %\quad 0\le t \le T,
$$
for some positive constant $C_2.$
The conclusion of the proof of Theorem \ref{thm obs} follows from the above inequality,  using Lemma \ref{cor obs}
to extend the above obser\-vability inequality from $[t_{j^*},t_{j^*+1}]$ 
%Theorem \ref{thm obs}, that is, for every $t\in
to $[0,T].$
\smartqed
\qed
\end{proof}

\begin{rmk}\em\label{rmk:sub}
By adapting the above proof, one could easily obtain an observability inequality for $u(\cdot,0)$ on $\Omega$,   requiring measurements just on the subboundary $\Sigma$ defined in (\ref{eq:subb}).
\end{rmk}

%\section*{Appendix%: proof of Proposition \ref{Pr1}
%}\label{SecPr1}
\section*{Appendix}
\addcontentsline{toc}{section}{Appendix}
In this appendix we prove Lemma \ref{Pr1}. 
\begin{proof}{(\it of Lemma \ref{Pr1}).}
Since $H\in Lip([0,T];\R^d)$ there exists $L>0$ such that 
%\begin{equation}\label{Lip}
$$\vert H(t)-H(s)\vert \le L|t-s|,\;\forall t,s\in[0,T].$$
%\end{equation}
Let us consider, for simplicity, a uniform partition $\{t_j\}_0^m$ of $[0,T].$ Let us set 
$$
\displaystyle \eta_j:=\frac{H(t_j)}{|H(t_j)|},\;j=0\ldots,m-1.
$$
For $t \in [t_j,t_{j+1}]$, $j=0, ..., m-1$, we deduce
 \begin{eqnarray}
H(t) \cdot \eta_j&=&\left(H(t)-H(t_j)\right)\cdot\eta_j+H(t_j)\cdot\eta_j\geq 
-|H(t)-H(t_j)|+|H(t_j)| \nonumber\\
&\ge&-L|t-t_j|+|H(t_j)|\ge -L\frac{T}{m}+|H(t_j)|,
\label{eq1}
\end{eqnarray}
and, since $|H(t)|\le\left|H(t)-H(t_j)\right|+|H(t_j)|,$ 
\begin{equation}\label{eq2}
\left|H(t_j)\right|\geq |H(t)|-|H(t)-H(t_j)|\ge|H(t)|-L|t-t_j|
\ge |H(t)|-L\frac{T}{m}.
\end{equation}
From (\ref{eq1}) and (\ref{eq2}), if we choose %the number of intervals of 
the uniform partition with %$m$ such that
 $m\ge \frac{2%\sqrt{2}%(\sqrt{2}+1)
 LT}{H_0(1-S_*)}$, where we recall that 
$\displaystyle H_0=\min_{t\in[0,T]}|H(t)|$, we obtain the conclusion, 
that is,
$$
H(t)\cdot
\frac{H(t_j)}{|H(t_j)|}\ge |H(t)|-2L\frac{T}{m}\ge %\frac{1}{\sqrt{2}}
S_*
|H(t)|,
\;\;\;\forall t\in[t_{j},t_{j+1}], \;\;\forall j=0,\ldots,m-1.
$$
\smartqed
\qed
\end{proof}

\begin{acknowledgement} This work was partially supported by Grant-in-Aid for 
Scientific Research (S) 15H05740 
and A3 Foresight Program Modeling and Computation of Applied Inverse 
Problemsh by Japan Society for the Promotion of Science.
The first and second author were visitor at The University of 
Tokyo in February 2018, supported by the above grant.\\
The third author was a Visiting Scholar at Rome in April 2018 supported by the University of Rome \lq\lq 
Tor Vergata''.
The third author was also visitor in July 2017 at the University of Naples Federico II, supported by the Department of Mathematics and Applications \lq\lq R. Caccioppoli'' of that University. \\
This work was supported also by the Istituto Nazionale di Alta Matematica 
(INdAM), through the GNAMPA Research Project 2017 \lq\lq 
Comportamento asintotico e controllo di equazioni di evoluzione non lineari'' 
(the coordinator: C. Pignotti). 
Moreover, this research was %has been 
performed within the framework of the GDRE CONEDP (European Research Group 
on \lq\lq Control of Partial Differential Equations'') issued by CNRS, INdAM 
and Universit\'e de Provence. This work was also supported by 
the research project of the University of Naples Federico II: \lq\lq Spectral 
and Geometrical Inequalities''.
\end{acknowledgement}

\end{document}